\documentclass[reqno]{amsart}

\usepackage{enumerate}
\usepackage{amsmath}%
\usepackage{amsfonts}%
\usepackage{amssymb}%
\usepackage{graphicx}
\usepackage{pgfplots}
\usepackage{mathrsfs}
\usepackage{hyperref}
\usepackage[margin=1in]{geometry}

\newtheorem{theorem}{Theorem}
\theoremstyle{plain}

\newtheorem{claim}[theorem]{Claim}

\newtheorem{conjecture}[theorem]{Conjecture}
\newtheorem{construction}[theorem]{Construction}
\newtheorem{corollary}[theorem]{Corollary}

\newtheorem{definition}[theorem]{Definition}

\newtheorem{fact}[theorem]{Fact}
\newtheorem{lemma}[theorem]{Lemma}

\newtheorem{proposition}[theorem]{Proposition}

\numberwithin{equation}{section}
\numberwithin{theorem}{section}
\numberwithin{case}{section}

\numberwithin{subcase}{case}

\def\B{\mathcal{B}}

\def\F{\mathcal{F}}
\def\G{\mathcal{G}}
\def\K{\mathcal{K}}

\def\M{\mathcal{M}}

\def\cP{\mathcal{P}}
\def\Q{\mathcal{Q}}
\def\R{\mathcal{R}}

\def \a{\alpha}
\def \e{\epsilon}
\def \r{\gamma}
\def \bfu{\mathbf{u}}
\def \bfi{\mathbf{i}}
\def \bfv{\mathbf{v}}
\def \bfw{\mathbf{w}}
\def\eg{\emph{e.g.}}
\def\ex{\text{ex}}
\def\coex{{\rm coex}}

\begin{document}

\title[]{Codegree conditions for tiling complete $k$-partite $k$-graphs and loose cycles}

\thanks{
The second author was supported by FAPESP (Proc. 2013/03447-6, 2014/18641-5, 2015/07869-8).
The third author is partially supported by NSF grants DMS-1400073 and DMS 1700622.}
\author{Wei Gao}
\address{Department of Mathematics and Statistics, Auburn University, Auburn, AL 36830}
\email[Wei Gao]{wzg0021@auburn.edu}
\author{Jie Han}
\address{Department of Mathematics, University of Rhode Island, 5 Lippitt Road, Kingston, RI, USA, 02881}
\email[Jie Han]{jie\_han@uri.edu}
\author{Yi Zhao}
\address
{Department of Mathematics and Statistics, Georgia State University, Atlanta, GA 30303} 
\email[Yi Zhao]{yzhao6@gsu.edu}

\begin{abstract} 
Given two $k$-graphs ($k$-uniform hypergraphs) $F$ and $H$, a perfect $F$-tiling (or $F$-factor) in $H$ is a set of vertex disjoint copies of $F$ that together cover the vertex set of $H$.
%A $k$-graph $F$ is $k$-partite if there exists a partition of $V(F)$ into $k$ parts such that every edge of $F$ contains exactly one vertex from each part.
For all complete $k$-partite $k$-graphs $K$, Mycroft proved a minimum codegree condition that guarantees a $K$-factor in an $n$-vertex $k$-graph, which is tight up to an error term $o(n)$.
%He also conjectured that the error term $o(n)$ should be replaced by a constant that only depends on $F$.
In this paper we improve the error term in Mycroft's result to a sub-linear term that relates to the Tur\'an number of $K$ when 
the differences of the sizes of the vertex classes of $K$ are co-prime. 
Furthermore, we find a construction which shows that our improved codegree condition is asymptotically tight in infinitely many cases thus disproving a conjecture of Mycroft. 
At last, we determine exact minimum codegree conditions for tiling $K^{(k)}(1, \dots, 1, 2)$ and tiling loose cycles thus generalizing the results of Czygrinow, DeBiasio, and Nagle, and of Czygrinow, respectively.
\end{abstract}

\maketitle

\section{Introduction}

Given $k\ge 2$, a $k$-uniform hypergraph (in short, \emph{$k$-graph}) is a pair $H=(V, E)$, where $V$ is a finite vertex set and $E$ is a family of $k$-element subsets of $V$. Given a $k$-graph $H$ and a set $S$ of $d$ vertices in $V(H)$, $1\le d\le k-1$, we denote by $\deg_H(S)$ the number of edges of $H$ containing $S$. The \emph{minimum $d$-degree $\delta _{d}
(H)$} of $H$ is the minimum of $\deg(S)$ over all $d$-subsets $S$ of $V(H)$.
Furthermore, the minimum 1-degree is usually referred as the \emph{minimum vertex degree} and the minimum $(k-1)$-degree is referred as the \emph{minimum collective degree (codegree)}.

As a natural extension of matching problems, (hyper)graph \emph{tiling} (alternatively called \emph{packing}) has received much attention in the last two decades (see \cite{KuOs-survey} for a survey). Given two (hyper)graphs $F$ and $H$, a \emph{perfect $F$-tiling}, or an \emph{$F$-factor}, of $H$ is a spanning subgraph of $H$ that consists of vertex disjoint copies of $F$. 
Here we are interested in minimum degree threholds that force perfect packings in hypergraphs.
Given a $k$-graph $F$ and an integer $n$ divisible by $|F|$, let ${\delta(n,F)}$ be the smallest integer $t$ such that every $n$-vertex $k$-graph $H$ with $\delta_{k-1}(H)\ge t$ contains a perfect $F$-tiling.

Perfect tilings for graphs are well understood.
In particular, extending the results of Hajnal and Szemer\'edi~\cite{HaSz} and Alon and Yuster~\cite{AY96} (see also~\cite{KSS-AY}), 
K\"uhn and Osthus~\cite{KuOs09} determined $\delta(n, F)$ for all graphs $F$, up to an additive constant, for sufficiently large $n$.

Over the last few years there has been a growing interest in obtaining degree conditions that force a perfect $F$-tiling in $k$-graphs for $k \geq 3$. 
In general, this appears to be much harder than the graph case (see a recent survey~\cite{zsurvey}). 
Let $K_4^3$ be the complete 3-graph on four vertices, and let $K_4^{3-}$ be the (unique) 3-graph on four vertices with three edges. 
Let $C_2^3$ be the unique 3-graph on four vertices with two edges. 
Lo and Markstr\"om \cite{LM1} proved that $\delta(n, K_4^3)=(1+o(1))3n/4$, and independently Keevash and Mycroft \cite{KM1} determined the exact value of $\delta(n, K_4^3)$  for sufficiently large $n$. 
In \cite{LM2} Lo and Markstr\"om proved that $\delta(n, K_4^{3-})=(1+o(1))n/2$. 
Very recently, Han Lo, Treglown and Zhao~\cite{HLTZ_K4} determined $\delta(n, K_4^{3-})$ exactly for large $n$. 
K\"uhn and Osthus \cite{KO} showed that $\delta(n, C_2^3)=(1+o(1))n/4$, and Czygrinow, DeBiasio, and Nagle \cite{CDN} determined $\delta(n, C_2^3)$ exactly for large $n$. 
Han and Zhao~\cite{HZ3} determined the exact minimum vertex degree threshold for perfect $C_2^3$-tiling for large $n$. 
With more involved arguments, Han, Zang, and Zhao~\cite{HZZ_tiling} determined the minimum {vertex} degree threshold for perfect $K$-tiling asymptotically for all complete $3$-partite $3$-graphs $K$.

Mycroft~\cite{My14} proved a general result on tiling $k$-partite $k$-graphs. To state his result, we need the following definitions. Let $F$ be a $k$-graph on a vertex set $U$ with at least one edge.
A \emph{$k$-partite realization} of $F$ is a partition of $U$ into vertex classes $U_1,\dots, U_k$ so that for any $e\in E(F)$ and $1\le j\le k$ we have $|e\cap U_j|=1$. 
We say that $F$ is \emph{$k$-partite} if it admits a $k$-partite realization.
Define
\[
\mathcal{S}(F):= \bigcup_{\chi} \{|U_1|,\dots, |U_k|\} \text{ and } \mathcal{D}(F):=\bigcup_{\chi} \{||U_i| - |U_j||: i, j\in [k]\},
\]
where in each case the union is taken over all $k$-partite realizations $\chi$ of $F$ into vertex classes $U_1,\dots, U_k$ of $F$.
Then $\gcd(F)$ is defined to be the greatest common divisor of the set $\mathcal{D}(F)$ (if $\mathcal{D}(F)=\{0\}$ then $\gcd(F)$ is undefined).
We also define
\[
\sigma(F):= \frac{\min_{S\in \mathcal{S}(F)}S}{|V(F)|},
\]
and thus in particular, $\sigma(F)\le 1/k$.
Mycroft~\cite{My14} proved the following:
\begin{equation} \label{eq:nk1}
\delta(n, F) \le \left\{\begin{array}{ll}
{n}/{2} +o(n) & \text{if } \mathcal{S}(F)=\{1\} \text{ or } \gcd(\mathcal{S}(F))>1;\\
\sigma(F)n +o(n)  & \text{if } \gcd(F)=1;\\
\max \{\sigma(F)n, n/p\}+o(n) & \text{if } \gcd(\mathcal{S}(F))=1 \text{ and } \gcd(F)=d>1, 
\end{array} \right.\end{equation} 
where $p$ is the smallest prime factor of $d$. 
Moreover, Mycroft~\cite{My14} showed that equality holds in~\eqref{eq:nk1} for all complete $k$-partite $k$-graphs $F$, as well as a wide class of other $k$-partite $k$-graphs.
Furthermore, he conjectured that the error terms in~\eqref{eq:nk1} can be replaced by a (sufficiently large) constant that depends only on $F$. 

\begin{conjecture}\cite{My14}\label{conj:Mycroft}
Let $F$ be a $k$-partite $k$-graph.
Then there exists a constant $C$ such that the error term $o(n)$ in~\eqref{eq:nk1} can be replaced by $C$.
\end{conjecture}

Let $K^{(k)}(a_1, \dots, a_k)$ denote the complete $k$-partite $k$-graph with parts of size $a_1, \dots, a_k$. In this paper we always assume that $a_1\le \cdots\le a_k$. Thus $\sigma(K^{(k)}(a_1, \dots, a_k)) = a_1/m$, where $m:=a_1+\cdots + a_k$. 
%In this paper we are mainly interested in the case when $K$ is a complete $k$-partite $k$-graph and $\gcd(K)=1$.
%Note that when $K$ is a complete $k$-partite $k$-graph, it has a unique $k$-partite realization, which we denote by $K:=K^{(k)}(a_1, \dots, a_k)$ for positive integers $a_1\le \cdots\le a_k$.
%In addition, if $\gcd(K)=1$, then we have $\gcd(a_2-a_1, a_3-a_2, \dots, a_k-a_{k-1})=1$.
%, i.e., there exist integers $\ell_1,\dots, \ell_{k-1}$ such that
%\[
%\ell_1 (a_2 - a_1) + \ell_2 (a_3 - a_1) + \cdots + \ell_{k-1} (a_k - a_{1}) =1.
%\]
The well-known \emph{space-barrier} (Construction~\ref{cons:lb0}) shows that 
\begin{equation}
\label{eq:dlb}
\delta(n, K^{(k)}(a_1, \dots, a_k))\ge \frac{a_1}{m} n.
\end{equation}
This shows that the second line of~\eqref{eq:nk1} is asymptotically best possible when $F= K^{(k)}(a_1, \dots, a_k)$ and $\gcd(F)=1$.

We first give a simple construction (Construction~\ref{cons:lb}) that strengthens the space-barrier. Applying this construction, we 
obtain the following proposition, whose Part (1) shows that Conjecture~\ref{conj:Mycroft} is false for all complete $k$-partite $k$-graphs $K$ with $\gcd(K)=1$ and $a_{k-1}\ge 2$.
Given two $k$-graphs $F$ and $H$, we call $H$ \emph{$F$-free} if $H$ does not contain $F$ as a subgraph.
The well-known \emph{Tur\'an number} ${\rm ex}(n, F)$ is the maximum number of edges in an $F$-free $k$-graph on $n$ vertices. Correspondingly, the \emph{codegree Tur\'an number} ${\rm coex}(n, F)$ is the maximum of the minimum codegree of an $F$-free $k$-graph on $n$ vertices. Note that $\coex(n, F) \binom{n}{k-1}/k \le \ex(n, F)$ because an $n$-vertex $k$-graph $H$ with $\delta_{k-1}(H)\ge  \coex(n, F)$ has at least $\coex(n, F) \binom{n}{k-1}/ k$ edges. 

\begin{proposition}\label{disprove}
Let $K:=K^{(k)}(a_1, \dots, a_k)$ such that $a_1\le \cdots \le a_k$ and $m=a_1+\cdots + a_k$. 
\begin{enumerate}
\item If $a_{k-1}\ge 2$, then $\delta(n, K) \ge a_1 n/m + (1-o(1))\sqrt{(m-a_1)n/m}$. 
\item If $a_1 =1$, then
\[
\delta(n, K) \ge \frac{n}{m} + {\rm coex} (\tfrac{m-1}{m}n+1, K).
\]
\end{enumerate}
\end{proposition}
%Then Thus Conjecture~\ref{conj:Mycroft} is false for all complete $k$-partite $k$-graphs $K$ with $\gcd(K)=1$ and $a_{k-1}\ge 2$.

Our main result sharpens the second case of~\eqref{eq:nk1} by using the Tur\'an number and the {Frobenius number}.
Given integers $0\le b_1\le \cdots\le b_k$ such that $\gcd(b_1,\dots, b_k)=1$, the \emph{Frobenius number}  $g(b_1,\dots, b_k)$ is the largest integer that cannot be expressed as $\ell_1 b_1+\cdots +\ell_k b_k$ for any nonnegative integers $\ell_1,\dots, \ell_k$. \footnote{The usual definition of Frobenius numbers requires that all $b_1, \dots, b_k$ are positive and distinct.}
By definition, $g(b_1,\dots, b_k)=-1$ if some $b_i=1$; otherwise $g(b_1,\dots, b_k)>0$. No general formula of $g(b_1,\dots, b_k)$ is known but it is known \cite{ErGr, Vitek75} that $g(b_1,\dots, b_k)\le (b_k-1)^2$.

\begin{theorem}\label{thm:main}
Let $k\ge 3$ and $K:=K^{(k)}(a_1, \dots, a_k)$ such that $a_1\le \cdots \le a_k$, $m=a_1+\cdots + a_k$ and $\gcd(K)=1$. 
Let $n\in m\mathbb{N}$ be sufficiently large.
Suppose $H$ is an $n$-vertex $k$-graph such that
\begin{equation}\label{eq:deg}
\delta_{k-1}(H)\ge \frac{a_1}{m}n + f(n) + C,
\end{equation}
where
\[
f(n) := \max_{1-C\le i\le 1} {{\rm ex}(\tfrac{m-a_1}{m}n +i, K)} k {\binom{\tfrac{m-a_1}{m}n+i}{k-1}}^{-1}
\]
and $C=g(a_2-a_1,\dots, a_k-a_1)+1$.
Then $H$ contains a $K$-factor.
\end{theorem}

A classical result of Erd\H{o}s \cite{erdos} states that given integers $k\ge 2$ and $1\le a_1\le \cdots \le a_k$, there exists $c$ such that for all sufficiently large $n$,
\begin{equation}\label{eq:erdos}
{\rm ex}(n, K^{(k)}(a_1,\dots, a_k)) \le c n^{k - 1/a_1 \cdots a_{k-1}}.
\end{equation}
This implies that $f(n)$ in Theorem~\ref{thm:main} is at most $O(n^{1-1/a_1 \cdots a_{k-1}})$, which is smaller than the error term $o(n)$ in~\eqref{eq:nk1}. 
Due to Proposition~\ref{disprove} (2), the term $f(n)$ in Theorem~\ref{thm:main} would be asymptotically tight if $a_1 = 1$ and ${\rm coex}(n, K) = (1- o(1)) \text{ ex}(n, K) k/ \binom{n}{k-1}$ (i.e., the extremal $k$-graph of $K$ is almost regular in terms of codegree).
%there is a $K$-free $k$-graph $G$ on $n$ vertices with $\delta_{k-1}(G)\ge (1- o(1)) \text{ ex}(n, K)/ \binom{n}{k-1}$ (thus $e(G)\ge (1- o(1)) \text{ ex}(n, K)$ and $G$ is almost regular in terms of codegree). % YZ added a_1=1.
%when ${\rm ex}(n, K)$ is (asymptotically) known and
%In view of Construction~\ref{cons:lb} (see Section 2), an immediate corollary of  is that if $a_1=\gcd(K)=1$ and there exists an extremal $k$-graph for ${\rm ex}(n, K)$ which is almost $d$-codegree-regular (i.e., any $(k-1)$-set has degree $(1+o(1))d$), then $\delta_{k-1}(n, K) = a_1 n/m + (1+o(1))d$.
Mubayi~\cite{Mubayi2002} determined ex$(n, K^{(k)}(1,\dots,1,2,t))$ asymptotically for all $t\ge 2$. 
Since the extremal $k$-graphs in this case is almost regular in terms of codegree, we obtain sharpened value of $\delta(n, K^{(k)}(1,\dots,1,2,t))$.
Moreover, Mubayi~\cite{Mubayi2002} also determined the order of magnitude of ex$(n, K^{(k)}(1,\dots,1,s,t))$ for $s\ge 3$ and $t\ge (s-1)!+1$.
%\footnote{Mubayi also determined ex$(n, K^{(k)}(1,\dots,1,3,3))$ asymptotically but we cannot apply it here because $\gcd(K^{(k)}(1,\dots,1,3,3))=2$.} 
This gives the correct order of magnitude of the second term of $\delta(n, K^{(k)}(1,\dots,1,s,t))$ for $s\ge 3$ and $t\ge (s-1)!+1$ such that $\gcd(s-1, t-s)=1$.

\begin{corollary}\label{thm:threshold}
Let $k\ge 3$. 
\begin{enumerate}[$(1)$]
\item For any $t\ge 2$,
\[
\delta(n, K^{(k)}(1,\dots,1,2,t))= \frac{n}{k+t} + (1+o(1))\sqrt{\frac{(t-1)(k+t-1)}{k+t}n};
\]

\item For any $s\ge 3$ and $t\ge (s-1)!+1$ such that $\gcd(s-1, t-s)=1$,
\[
\delta(n, K^{(k)}(1,\dots,1,s,t))= \frac{n}{k+s+t-2} + \Theta (n^{1-1/s}).
\]
\end{enumerate}
\end{corollary}

\medskip

If $K:=K^{(k)}(a_1, \dots, a_k)$ satisfies $\gcd(K)=1$ and $a_{k-1}=1$, then $a_k=2$ and consequently, $K= K^{(k)}(1,\dots, 1, 2)$. In this case ${\rm ex}(n, K) \le \binom{n}{k-1}/k$ because in a $K$-free $k$-graph, every $(k-1)$-set has degree at most $1$. Moreover, $C=g(0,\dots, 0,1)+1 =0$ in this case. Theorem~\ref{thm:main} thus gives that $\delta(n, K)\le n/(k+1)+1$. By a more careful analysis on the proof of Theorem~\ref{thm:main}, we are able to determine $\delta(n, K^{(k)}(1,\dots, 1,2))$ exactly (for sufficiently large $n$).

\begin{theorem}\label{thm:112}
Given $k\ge 3$, let $n\in (k+1)\mathbb{Z}$ be sufficiently large. Then
\begin{equation*}
\delta(n, K^{(k)}(1,\dots, 1,2))=\begin{cases}
 \frac{n}{k+1}+1 &\text{ if } k-i \mid\binom{n'-i}{k-1-i} \text{ for all } 0\le i\le k-2; \\
 \frac{n}{k+1} &\text{ otherwise},
\end{cases}
\end{equation*}
where $n' = \frac{k n}{k+1}+1$.
\end{theorem}
A \emph{Steiner system} $S(t, k, n)$ is an $n$-vertex $k$-uniform hypergraph in which every set of $t$ vertices has degree exactly $1$. The divisibility conditions in Theorem~\ref{thm:112} are necessary for the existence of $S(k-1, k, n')$. 
Our proof of Theorem~\ref{thm:112} applies a recent breakthrough of Keevash~\cite{Keevash_design}, who showed that these divisibility conditions are also sufficient for the existence of a Steiner system $S(k-1, k, n')$ for sufficiently large $n'$.

When $k=3$, the divisibility conditions in Theorem~\ref{thm:112} reduce to $8\mid n$. Since $K^{(3)}(1, 1, 2)= C_2^3$, Theorem~\ref{thm:112} gives the aforementioned result of Czygrinow, DeBiasio and Nagle \cite{CDN}.
When $k$ is even, the divisibility conditions in Theorem~\ref{thm:112} always fail and consequently, $\delta(n, K)= n/(k+1)$.
To see this, letting $i=k-2$, we have $k-i=2$ and $\binom{n'-i}{k-1-i} = n' - k+2 = \frac{k n}{k+1} - k + 3$.
When $k$ is even, $\frac{k n}{k+1} - k + 3$ is odd and thus $k-i\nmid \binom{n'-i}{k-1-i}$.

\medskip
Our last result is on tiling loose cycles.
For $k\ge 3$ and  $s>1$, a loose cycle of length $s$, denoted $C_s^k$, is a $k$-graph with $s(k-1)$ vertices $1,\dots, s(k-1)$ and $s$ edges $\{j(k-1)+1,\dots, j(k-1)+k\}$ for $0\le j <s$, where we regard $s(k-1)+1$ as $1$. 
%For example, $C^3_2= K(1, 1, 2)$. 
It is easy to see that $\gcd(C_s^k)=1$ unless $s=k=3$ (see Proposition~\ref{propcsk}). 
R\"odl and Ruci\'nski~\cite[Problem 3.15]{RR} asked for the value of $\delta(n, C_s^3)$.
Mycroft~\cite{My14} determined $\delta(n, C_s^k)$ asymptotically for all $s\ge 2$ and $k\ge 3$.
Recently, Gao and Han~\cite{GaHa} show that $\delta(n, C_3^3)=n/6$ and independently Czygrinow~\cite{Czy_cycle} determined $\delta(n, C_s^3)$ for all $s\ge 3$.
By modifying the proof of Theorem~\ref{thm:main}, we determine the exact value of $\delta(n, C_s^k)$ for $k\ge 4$ and $s\ge 2$.
%To simplify some notation and arguments, we assume that $k\ge 4$ and $s\ge 2$ though our proof naively extends to the case $k=3$ and $s\ge 4$. The only different case is when $k=s=3$, which we have addressed in~\cite{GaHa} separately (indeed, the authors of \cite{My14} and~\cite{Czy_cycle} also addressed this special case separately).

\begin{theorem}\label{thm:cyc_pac}
Given $k\ge 4$ and $s\ge 2$, let $n\in s(k-1)\mathbb{N}$ be sufficiently large.
Suppose $H$ is an $n$-vertex $k$-graph such that $\delta_{k-1}(H)\ge \frac{\lceil s/2 \rceil}{s(k-1)}n$.
Then $H$ contains a $C_s^k$-factor.
\end{theorem}

Construction~\ref{cons:lb0} shows that the codegree condition in Theorem~\ref{thm:cyc_pac} is sharp.

The rest of the paper is organized as follows. % YZ removed give Construction~\ref{cons:lb} and
We prove Proposition~\ref{disprove} and Corollary~\ref{thm:threshold} in Section~2. 
Next we discuss proof ideas and give auxiliary lemmas and use them to prove Theorems~\ref{thm:main},~\ref{thm:112}, and~\ref{thm:cyc_pac} in Section~3. 
We prove the auxiliary lemmas in Sections 4--6.

\section{Proof of Theorem~\ref{thm:threshold}}

The following well-known construction is often called the \emph{space barrier} (for tiling problems).
Given a $k$-graph $F$, let $\tau(F)$ be the smallest size of a \emph{vertex cover} of $F$, namely, a set that meets each edge of $F$.
Trivially $\tau(K^{(k)}(a_1, \dots, a_k)) = a_1$. We also have $\tau(C_s^k) \ge \lceil s/2 \rceil$ because $C_s^k$ has $s$ edges and every vertex of $C_s^k$ has degree at most two.\footnote{We also know $\tau(C_s^k) \le \lceil s/2 \rceil$ from Proposition~\ref{propcsk}.} % YZ revised the footnote.

\begin{construction}\label{cons:lb0}
Fix a $k$-graph $F$ of $m$ vertices. Let $H_0=(V, E)$ be an $n$-vertex $k$-graph such that $V=A\cup B$ with $|A| = \tau(F) n/m - 1$ and $|B|=n-|A|$, and $E$ consists of all $k$-sets that intersect $A$. We have $\delta_{k-1}(H_0) = |A| = \tau(F) n/m - 1$.
\end{construction}
Since each copy of $F$ in $H_0$ contains at least $\tau(F)$ vertices in $A$, $H_0$ does not contain a perfect $F$-tiling.
We slightly strengthen Construction~\ref{cons:lb0} as follows. 

\begin{construction}\label{cons:lb}
Let $K:=K^{(k)}(a_1, \dots, a_k)$ such that $a_1\le \cdots \le a_k$ and $m=a_1+\cdots + a_k$.
Let $H_1=(V, E)$ be an $n$-vertex $k$-graph as follows. Let $V=A\cup B$ such that $|A| = a_1 n/m - 1$ and $|B|=n-|A|$.
Let $G$ be a $k$-graph on $B$ which is $K^{(k)}(b_1, b_2, \dots, b_k)$-free for all $1\le b_1\le \dots \le b_k$ such that $\sum_{i\in [k]} b_i = m - a_1+1$ and $b_i\le a_i$ for $i\in [k]$. 
Let $E$ be the union of $E(G)$ and the set of all $k$-tuples that intersect $A$, and thus $\delta_{k-1}(H_0) = |A| + \delta_{k-1}(G) =a_1 n/m + \delta_{k-1}(G) - 1$. 
\end{construction}

In Construction~\ref{cons:lb}, no $m$-set with at most $a_1-1$ vertices in $A$ spans a copy of $K$. 
Therefore each copy of $K$ in $H_0$ contains at least $a_1$ vertices in $A$ and consequently, $H_0$ does not contain a perfect $K$-tiling.

Now we give a construction of Mubayi~\cite{Mubayi2002}. 
Given $t\ge 2$, suppose that $q$ is a prime number such that $q\equiv 1 \bmod t-1$. Let $n_0=(q-1)^2/(t-1)$.
Let $\mathbf{F}$ be the $q$-element finite field, and let $S$ be a (multiplicative) subgroup of $\mathbf{F}\setminus \{0\}$ of order $t-1$.
We define a $k$-graph $G_0$ whose vertex set consists of all equivalence classes in $(\mathbf{F}\setminus \{0\})\times (\mathbf{F}\setminus \{0\})$, where $(a, b)\sim (x,y)$ if there exists $s\in S$ such that $a=s x$ and $b=s y$.
The class represented by $(a,b)$ is denoted by $\langle a, b \rangle$.
A set of $k$ distinct classes $\langle a_i, b_i \rangle$ $(1\le i\le k)$ forms an edge in $G_0$ if
\[
\prod_{i=1}^k a_i + \prod_{i=1}^k b_i \in S.
\]
It is easily observed that this relation is well-defined, and $\delta_{k-1}(G_0)\ge q-k$.
Moreover, as shown in~\cite{Mubayi2002}, $G_0$ is $K^{(k)}(1,\dots,1,2,t)$-free.

To extend this construction, we use the fact that for any $\e>0$ and sufficiently large $n$, there exists a prime $q$ such that $q\equiv 1 \bmod t-1$ and $n\le (q-1)^2/(t-1) \le (1+\e/3)n$ (see \cite{HuIw}). 
Let $G_0$ be the $k$-graph on $(q-1)^2/(t-1)$ vertices defined above.
To obtain a $K^{(k)}(1,\dots,1,2,t)$-free $k$-graph $G$ on $n$ vertices, we delete a random set $T$ of order $(q-1)^2/(t-1) - n$ from $G_0$ and let $G:=G_0\setminus T$.
Since the expected value of the codegree survived is at least $(q-k)/(1+\e/3)$, standard concentration results (e.g., Chernoff's bound) show that $\delta_{k-1}(G) \ge (1-\e)\sqrt{(t-1)n}$ with positive probability.
We summarize this construction together with the result on ${\rm ex}(n, K^{(k)}(1,\dots, 1,2,t))$ from~\cite{Mubayi2002} in the following proposition.

\begin{proposition}\cite{Mubayi2002}\label{prop:lb}
For any $t\ge 2$, we have $\coex(n, K^{(k)}(1,\dots,1,2,t))= (1+ o(1))\sqrt{(t-1) n}$,
%Given $t\ge 2$ and $\e>0$ there exists $n_0\in \mathbb{N}$ such that for all $n\ge n_0$, there exists an $n$-vertex $k$-graph $G$ which is $K^{(k)}(1,\dots,1,2,t)$-free and $\delta_{k-1}(G) = (1-\e)\sqrt{(t-1) n}$.
and ${\rm ex}(n, K^{(k)}(1,\dots, 1,2,t)) = (1+o(1))\frac{\sqrt{t-1}}{k!} n^{k-1/2}$.
\end{proposition}

For integers $s\ge 3$ and $t\ge (s-1)!+1$, a more involved construction in~\cite{Mubayi2002} shows there exists a $K^{(k)}(1,\dots,1,s,t)$-free $k$-graph of order $q^s - q^{s-1}$ for some prime number $q$ with the desired minimum codegree.
We omit the detail of this construction and note that the construction can be extended to all sufficiently large $n$ as above.

\begin{proposition}\cite{Mubayi2002}\label{prop:lb2}
Given $s\ge 3$ and $t\ge (s-1)!+1$, we have $\coex(n, K^{(k)}(1,\dots,1,s,t))= \Theta(n^{1-1/s})$, 
%Given $s\ge 3$ and $t\ge (s-1)!+1$, $\e>0$, there exists $n_0\in \mathbb{N}$ such that for all $n\ge n_0$, there exists an $n$-vertex $k$-graph $G$ which is $K^{(k)}(1,\dots,1,s,t)$-free and $\delta_{k-1}(G) = (1-\e) n^{1-1/s}$.
and ${\rm ex}(n, K^{(k)}(1,\dots, 1,s,t)) = \Theta( n^{k-1/s})$.
\end{proposition}

\begin{proof}[Proof of Proposition~\ref{disprove}]
Assume $K:=K^{(k)}(a_1, \dots, a_k)$ such that $a_1\le \cdots \le a_k$, $m=a_1+\cdots + a_k$ and $a_{k-1}\ge 2$.
We will show that for any choice of $b_1, b_2, \dots, b_k$ such that $\sum_{i\in [k]} b_i = m - a_1+1$ and $b_i\le a_i$ for $i\in [k]$, we have $b_{k-1}\ge 2$ (thus $K^{(k)}(b_1, b_2, \dots, b_k)$ contains $K^{(k)}(1, \dots, 1, 2,2)$ as a subgraph). 
Then Proposition~\ref{disprove} (1) follows from putting the $k$-graph $G$ given by Proposition~\ref{prop:lb} with $t=2$ into Construction~\ref{cons:lb}. 
To see why $b_{k-1}\ge 2$, first assume that $a_1=1$.
In this case $b_i=a_i$ for all $i\in [k]$. Since $a_{k-1}\ge 2$, we have $b_{k-1}\ge 2$.
Second assume that $a_1\ge 2$.
If $b_{k-1} =1$, then $b_1 = \dots = b_{k-1}=1$ and consequently, $\sum_{i\in [k]} a_i - \sum_{i\in [k]} b_i \ge (k-1)(a_1-1) > a_1 - 1$, a contradiction. Thus $b_{k-1}\ge 2$.

Proposition~\ref{disprove} (2) follows from Construction~\ref{cons:lb} immediately because $a_1 = 1$ implies that $a_i = b_i$ for $i\in [k]$. 
\end{proof}

\begin{proof}[Proof of Corollary~\ref{thm:threshold}]
The upper bounds in Corollary~\ref{thm:threshold} $(1)$ and $(2)$ follow from Theorem~\ref{thm:main} and the results on the Tur\'an numbers from Propositions~\ref{prop:lb} and~\ref{prop:lb2}.
The lower bounds follow from Proposition~\ref{disprove} (2) and the results on the codegree Tur\'an numbers from Propositions~\ref{prop:lb} and~\ref{prop:lb2}.
%Then putting the $k$-graph $G$ given by Proposition~\ref{prop:lb} into Construction~\ref{cons:lb} shows the lower bound in $(1)$; and putting the $k$-graph $G$ given by Proposition~\ref{prop:lb2} into Construction~\ref{cons:lb} shows the lower bound in $(2)$. 
\end{proof}

\section{Proof ideas and lemmas}

Mycroft's proofs~\cite{My14} use the newly developed Hypergraph Blow-up Lemma by Keevash~\cite{Keevash_blowup}.
Instead, our proofs include several new ingredients, which allow us to obtain a better bound by a much shorter proof.
First, to obtain \emph{exact} results, we separate the proof into a \emph{non-extremal} case and an \emph{extremal} case and deal with them separately.
The proof of the non-extremal case utilizes the \emph{lattice-based absorbing method} developed recently by the second author~\cite{Han14_poly}, which builds on the \emph{absorbing method} initiated by R\"odl, Ruci\'nski and Szemer\'edi~\cite{RRS06}.
In order to find an almost perfect $K$-tiling, we use the so-called \emph{fractional homomorphic tiling}, which was used by Bu\ss, H\`an and Schacht in~\cite{BHS}, together with the weak regularity lemma for hypergraphs.
At last, we deal with the extremal case by careful analysis.

Now we give our lemmas.
Throughout the paper, we write $\alpha \ll \beta \ll \gamma$ to mean that we can choose the positive constants
$\alpha, \beta, \gamma$ from right to left. More
precisely, there are increasing functions $f$ and $g$ such that, given
$\gamma$, whenever we choose some $\beta \leq f(\gamma)$ and $\alpha \leq g(\beta)$, the subsequent statement holds. 
Hierarchies of other lengths are defined similarly.

\begin{lemma}[Absorbing Lemma]\label{lem:absorb}
Let $k\ge 3$ and $K:=K^{(k)}(a_1,\dots, a_k)$ such that $a_1\le \cdots \le a_k$, $m=a_1+\cdots + a_k$ and $\gcd(K)=1$.
Suppose $\r' \ll \r \ll \rho, 1/m$ and $n$ is sufficiently large.
If $H$ is an $n$-vertex $k$-graph such that $\delta_{k-1}(H)\ge \rho n$, then there exists a vertex set $W\subseteq V(H)$ with $|W|\leq \r n$ such that for any vertex set $U\subseteq V(H)\setminus W$ with $|U|\leq \r' n$ and $|U|\in m \mathbb{Z}$, both $H[W]$ and $H[U\cup W]$ contain $K$-factors.
\end{lemma}

We say $H$ is \emph{$\xi$-extremal} if there exists a set $B\subseteq V(H)$ of $\lfloor(1- \sigma(K))n\rfloor$ vertices such that $e(B)\le \xi \binom{|B|}{k}$.
In the following lemma we do not need the assumption $\gcd(K)=1$, instead we assume that $a_1<a_k$ (which is necessary for $\gcd(K)=1$).
Note that the $a_1=a_k$ (i.e., $a_1=\cdots=a_k$) case has been solved in~\cite[Lemma 2.4]{GaHa}.

\begin{lemma}[$K$-tiling Lemma] \label{lemK}
Let $k\ge 3$ and $K:=K^{(k)}(a_1,\dots, a_k)$ such that $a_1\le \cdots \le a_k$, $a_1<a_k$, $m=a_1+\cdots + a_k$.
For any $\a, \r, \xi >0$ such that $\r\ll 1/m$ and $\xi\ge 5bk^2 \r$, there exists an integer $n_0$ such that the following holds. If $H$ is a $k$-graph on $n>n_0$ vertices with $\delta_{k-1}(H)\ge (a_1/m - \gamma)n$,
then $H$ has a $K$-tiling that covers all but at most $\a n$ vertices unless $H$ is $\xi$-extremal.
\end{lemma}

Finally we give the extremal cases for Theorems~\ref{thm:main},~\ref{thm:112} and~\ref{thm:cyc_pac}, respectively.

\begin{theorem}\label{lemE}
Given $k\ge 3$, let $K:=K^{(k)}(a_1,\dots, a_k)$ such that $a_1\le \cdots \le a_k$, $m=a_1+\cdots + a_k$ and $\gcd(K)=1$.
Suppose $1/n\ll \xi \ll 1/m$ and $n\in m\mathbb{N}$.
If $H$ is an $n$-vertex $k$-graph which is $\xi$-extremal and satisfies \eqref{eq:deg}, then $H$ contains a $K$-factor.
\end{theorem}

\begin{theorem}\label{lemE3}
Given $k\ge 3$, let $1/n\ll \xi \ll 1/k$ such that $n\in (k+1)\mathbb{N}$.
Suppose $H$ is an $n$-vertex $k$-graph that is $\xi$-extremal.
Then $H$ contains a $K^{(k)}(1,\dots, 1,2)$-factor if either of the following holds:
\begin{enumerate}[(i)]
\item $\delta_{k-1}(H)\ge n/(k+1)+1$;
\item $\delta_{k-1}(H)\ge n/(k+1)$ and $k-i \nmid\binom{n'-i}{k-1-i}$ for some $0\le i\le k-2$ and $n' = \frac{k n}{k+1}+1$.
\end{enumerate}
\end{theorem}

\begin{theorem}\label{lemE2}
Given $k\ge 4$ and $s\ge 2$, let $1/n\ll \xi \ll 1/s, 1/k$ such that $n\in s(k-1)\mathbb{N}$.
Suppose $H$ is an $n$-vertex $k$-graph with $\delta_{k-1}(H)\ge \frac{\lceil s/2 \rceil}{s(k-1)}n$.
If $H$ is $\xi$-extremal, then $H$ contains a $C_s^k$-factor.
\end{theorem}

\begin{proof}[Proofs of Theorems~\ref{thm:main},~\ref{thm:112} and~\ref{thm:cyc_pac}]
We first prove Theorem~\ref{thm:main}.
Let $k\ge 3$ and $K:=K^{(k)}(a_1, \dots, a_k)$ such that $a_1\le \cdots \le a_k$, $m=a_1+\cdots + a_k$ and $\gcd(K)=1$. 
Suppose $1/n\ll \r'\ll \r \ll \xi \ll 1/m$ and $n\in m\mathbb{N}$.
Suppose $H$ is an $n$-vertex $k$-graph satisfying \eqref{eq:deg}.
If $H$ is $\xi$-extremal, then $H$ contains a $K$-factor by Theorem~\ref{lemE}.
Otherwise we apply Lemma~\ref{lem:absorb} and find an absorbing set $W$ in $V(H)$ of size at most $\r n$ which has the absorbing property.
Let $H':=H\setminus W$ and $n'=|V(H')| \ge (1-\r)n$.
Note that $m\mid n'$.
If $H'$ is $({\xi}/2)$-extremal, then there exists a vertex subset $B'$ in $V(H')$ of order $\lfloor (1-\sigma(K))n' \rfloor$ such that $e_{H'}(B')\le \frac{\xi}2 \binom{|B'|}{k}$.
Thus by adding to $B'$ at most $n-n'\le \r n$ vertices, we get a set $B$ of exactly $\lfloor (1-\sigma(K))n \rfloor$ vertices in $V(H)$ with % YZ rewrote the line below
\[
e_H(B) \le e_{H'}(B') + \r n \cdot \binom{n-1}{k-1} \le \frac{\xi}{2} \binom{|B'|}{k} + k \r \binom{n}{k} \le \xi \binom{|B|}{k}
\]
because the choice of $\r$.
This means that $H$ is $\xi$-extremal, a contradiction.
We thus assume that $H'$ is not $(\xi/2)$-extremal.
By applying Lemma~\ref{lemK} on $H'$ with $\xi/2$ and $\a=\r'$, we obtain a $K$-tiling $M$ that covers all but a set $U$ of at most $\r' n$ vertices.
By the absorbing property of $W$, $H[W\cup U]$ contains a $K$-factor and together with the $K$-tiling $M$ we obtain a $K$-factor of $H$.

The proof of Theorem~\ref{thm:cyc_pac} is the same except that we replace $K$ by $C_s^k$ and replace Theorem~\ref{lemE} by Theorem~\ref{lemE2} (here we apply Lemma~\ref{lemK} with the $k$-partite $k$-graph given by Proposition~\ref{propcsk}).
Similarly, after replacing Theorem~\ref{lemE} by Theorem~\ref{lemE3}, the arguments above prove the upper bounds in Theorem~\ref{thm:112}.
To see the lower bounds, we know $\delta(n, K^{(k)}(1,\dots, 1,2))\ge n/(k+1)$ from \eqref{eq:dlb}.
Let $n' = \frac{k n}{k+1}+1$.
If $k-i\mid \binom{n'-i}{k-1-i}$ for all $0\le i\le k-2$, then the result of Keevash~\cite{Keevash_design} implies that the Steiner system $S(k-1, k, n')$ exists, in other words, $\coex(n', K^{(k)}(1,\dots, 1,2))=1$.
Then the lower bound $\delta(n, K^{(k)}(1,\dots, 1,2))\ge n/(k+1)+1$ follows from Proposition~\ref{disprove} (2).
\end{proof}

\section{Proof of the Absorbing Lemma}

The following simple proposition will be useful.
\begin{proposition} \label{prop:deg}
Let $H$ be a $k$-graph. If $\delta_{k-1}(H)\geq x n$ for some $0\le x\le 1$, then $\delta_{1}(H)\geq x\binom{n- 1}{k- 1}$.
\end{proposition}

The following concepts were introduced by Lo and Markstr\"om~\cite{LM1}.
Given a $k$-graph $F$ of order $m$, $\beta>0$, $i \in \mathbb{N}$, we say that two vertices $u, v$ in a $k$-graph $H$ on $n$ vertices are \emph{$(F, \beta ,i)$-reachable (in $H$)} if and only if there are at least $\beta n^{im-1}$ $(im-1)$-sets $W$ such that both $H[\{u\} \cup W]$ and $H[\{v\} \cup W]$ contain $F$-factors.
A vertex set $A$ is \emph{$(F, \beta ,i)$-closed in $H$} if every pair of vertices in $A$ are $(F, \beta ,i)$-reachable in $H$.
For $x\in V(H)$, let $\tilde{N}_{F, \beta, i}(x)$ be the set of vertices that are $(F, \beta, i)$-reachable to $x$ in $H$.

We use the following lemma in~\cite{HT} which gives us a partition $\cP = \{V_1, \dots, V_r\}$ on $H$ such that for any $i\in [r]$, $V_i$ is $(F, \beta, 2^{c-1})$-closed. 

\begin{lemma}[\cite{HT}, Lemma 6.3]\label{lem:partition}
Let $c, k, m\ge 2$ be integers and suppose $1/n\ll \beta \ll \a \ll 1/c, \delta', 1/m$.
%, there exists constant $\beta>0$ such that the following holds for all sufficiently large $n$. 
Let $F$ be an $m$-vertex $k$-graph.
Assume an $n$-vertex $k$-graph $H$ satisfies that $|\tilde{N}_{F, \a, 1}(v)| \ge \delta' n$ for any $v\in V(H)$ and every set of $c+1$ vertices in $V(H)$ contains two vertices that are $(F, \a, 1)$-reachable. 
Then we can find a partition $\cP$ of $V(H)$ into $V_1,\dots, V_r$ with $r\le \min\{c, 1/\delta' \}$ such that for any $i\in [r]$, $|V_i|\ge (\delta' - \a) n$ and $V_i$ is $(F, \beta, 2^{c-1})$-closed in $H$.
\end{lemma}

Let $\cP = \{V_1, \dots, V_r\}$ be a vertex partition of $H$.
The \emph{index vector} $\mathbf{i}_{\cP}(S)\in \mathbb{Z}^r$ of a subset $S\subset V$ with respect to $\cP$ is the vector whose coordinates are the sizes of the intersections of $S$ with each part of $\cP$, i.e., $\mathbf{i}_{\cP}(S)_{V_i}=|S\cap V_i|$ for $i\in [r]$. 
We call a vector $\mathbf{i}\in \mathbb{Z}^r$ an \emph{$s$-vector} if all its coordinates are nonnegative and their sum equals to $s$. 
Given a $k$-graph $F$ of order $m$ and $\mu>0$, an $m$-vector $\mathbf{v}$ is called a \emph{$\mu$-robust $F$-vector} if there are at least $\mu n^m$ copies $C$ of $F$ in $H$ satisfying $\mathbf{i}_\cP(V(C))=\mathbf{v}$. 
Let  $I_{\cP, F}^{\mu}(H)$ be the set of all $\mu$-robust $F$-vectors.
For $j\in [r]$, let $\mathbf{u}_j\in \mathbb{Z}^r$ be the $j$th \emph{unit vector}, namely, $\bfu_j$ has 1 on the $j$th coordinate and 0 on other coordinates.
A \emph{transferral} is a vector of form $\bfu_i - \bfu_j$ for some distinct $i,j\in [r]$.
Let $L_{\cP, F}^{\mu}(H)$ be the lattice (i.e., the additive subgroup) generated by $I_{\cP, F}^{\mu}(H)$.

To prove Lemma~\ref{lem:absorb}, our main tool is Lemma~\ref{lem:partition} together with the following results. The next proposition is a simple counting result that follows from~\eqref{eq:erdos}.
%Note that in the following two results, the constants are not presented in the hierarchy, as we will need to retrieve $t_0$ in the proof of Lemma~\ref{lem:absorb}.

\begin{proposition}\label{supersaturation}
Given integers $k, r_0, a_1,\dots, a_k \in \mathbb{N}$, suppose that $1/n\ll \mu\ll \eta, 1/k, 1/r_0, 1/a_1,\dots, 1/a_k$. 
Let $H$ be a $k$-graph on $n$ vertices with a vertex partition $V_1 \cup \dots \cup V_r$ where $r\le r_0$. Suppose $i_1, \dots, i_k\in [r]$ and $H$ contains at least $\eta {n}^{k}$ edges $e=\{ v_1, \dots, v_k \}$ such that $v_1\in V_{i_1}$, $\dots, v_k\in V_{i_k}$. Then $H$ contains at least $\mu {n}^{a_1+\cdots + a_k}$ copies of $K^{(k)}(a_1,\dots, a_k)$ whose $j$th part is contained in $V_{i_j}$ for all $j\in [k]$.
\end{proposition}

We use the following result in \cite{HZZ_tiling}, which says that $V(H)$ is closed when all the transferrals of $\cP$ are present. We state it in a less general form, namely, we omit the trash set $V_0$ in its original form. % YZ removed roughly speaking

\begin{lemma}[\cite{HZZ_tiling}, Lemma 3.9]\label{lattice}
Let $i_0, k, r_0>0$ be integers and let $F$ be an $m$-vertex $k$-graph.
Suppose $1/n\ll1/i_0', \beta' \ll \e, \beta, \mu$ such that $i_0'\in \mathbb{Z}$.
Let $H$ be a $k$-graph on $n$ vertices with a partition $\cP=\{V_1, \dots, V_r\}$ such that $r\le r_0$ and for each $j\in [r]$, $|V_j|\ge \e n$ and $V_j$ is $(F,\beta,i_0)$-closed in $H$. 
If $\bfu_j - \bfu_l\in L_{\cP, F}^{\mu}(H)$ for all $1\leq j< l\leq r$, then $V(H)$ is $(F,\beta',i_0')$-closed in $H$.
\end{lemma}

The following lemma of Lo and Markstr\"om provides the desired absorbing set when $V(H)$ is closed. % YZ added Lo and Markstr\"om

\begin{lemma}[\cite{LM1}, Lemma 1.1]\label{LM1.1}
Let $m$ and $i$ be positive integers and let $F$ be an $m$-vertex $k$-graph. 
Suppose $1/n\ll \r' \ll \beta, 1/m, 1/i$ and $H$ is an $(F, \beta, i)$-closed $k$-graph of order $n$. 
Then there exists a vertex set $W\subseteq V(H)$ with $|W|\leq \beta n$ such that for any vertex set $U\subseteq V(H)\setminus W$ with $|U|\leq \r' n$ and $|U|\in m \mathbb{Z}$, both $H[W]$ and $H[U\cup W]$ contain $F$-factors.
\end{lemma}

We need another lemma from \cite{LM1}.

\begin{lemma}[\cite{LM1}, Lemma 4.2]\label{LM4.2}
Let $k \ge 2$ be an integer and $F$ be an $m$-vertex $k$-partite $k$-graph. 
Suppose $1/n\ll \a \ll \r$. 
For any $k$-graph $H$ of order $n$, two vertices $x, y \in V(H)$ are $(F, \a, 1)$-reachable to each other if the number of $(k-1)$-sets $S \in N(x) \cap N(y)$ with $|N(S)| \ge \r n$ is at least $\r^2  {n \choose k-1}$.
\end{lemma}

\begin{proof}[Proof of Lemma~\ref{lem:absorb}]
Let $k\ge 3$ and $K:=K^{(k)}(a_1,\dots, a_k)$ such that $a_1\le \cdots \le a_k$, $m=a_1+\cdots + a_k$ and $\gcd(K)=1$.
Given $\rho>0$, we select other parameters as follows.
We pick $0<\eta, \r \ll \rho$ such that $(\lfloor 1/\rho \rfloor+1)\rho > 1+(\lfloor 1/\rho \rfloor+1)^2\r$. % YZ defined \eta, \r together and removed \ll 1/m because in my opinion, \r \ll \rho already means \r < \rho / k^k -- this is what we need in (4.1).
When applying Proposition~\ref{supersaturation} with $r_0=\lfloor 1/\rho \rfloor$ and $\eta$, we get $0<\mu\ll\eta$.
When applying Lemma~\ref{LM4.2}, we get $0<\a \ll \r$.
When applying Lemma~\ref{lem:partition} with $\a$, $c={\lfloor}1/\rho{\rfloor}$, and $\delta'=\rho-\r$, we get $\beta>0$.
When applying Lemma~\ref{lattice} with $i_0=2^{{\lfloor}1/\rho{\rfloor}-1}$, $r_0={\lfloor}1/\rho{\rfloor}$, $\e=\rho-2\r$, $\beta$, and $\mu$, we get $\beta'$ and an integer $i_0'$.
Finally, when applying Lemma~\ref{LM1.1} with $\beta'$ and $i_0'$, we get $\r'>0$.

Let $n$ be a sufficiently large integer and let $H$ be a $k$-graph on $n$ vertices such that $\delta_{k-1}(H)\ge \rho n$.
By Proposition~\ref{prop:deg}, we have $\delta_1(H)\ge \rho{n-1 \choose k-1}$.
First, for every $v\in V(H)$, let us bound $|\tilde{N}_{K, \a, 1}(v)|$ from below. % YZ revised.
Given any $(k-1)$-set $S$, we have $|N(S)| \ge \rho n$ because $\delta_{k-1}(H)\ge \rho n$. Then by Lemma \ref{LM4.2}, for any distinct $u, v \in V(H)$, $u \in \tilde{N}_{K, \a, 1}(v)$ if $|N(u) \cap N(v)| \ge \r^2  {n \choose k-1}$. By double counting, we have
\[
{\sum_{S\in N(v)}|N(S)|} < |\tilde{N}_{K, \a, 1}(v)| \cdot |N(v)| + n \cdot \r^2  {n \choose k-1}.
\] 
Note that $|N(v)| \ge \delta_1(H)\ge \rho{n-1 \choose k-1}$. Together with $\r \ll \rho$, we get 
\begin{equation}\label{eq:NK}
|\tilde{N}_{K, \a, 1}(v)| > \rho n - \frac{\r^2n^k}{|N(v)|} \ge (\rho-\r)n.
\end{equation}

Next we claim that every set of ${\lfloor}1/\rho{\rfloor}+1$ vertices in $V(H)$ contains two vertices that are $(K, \a, 1)$-reachable.
Indeed, since $\delta_{1}(H)\ge \rho{n-1 \choose k-1}$, the degree sum of any $\lfloor 1/\rho \rfloor+1$ vertices is at least $({\lfloor}1/\rho{\rfloor}+1)\rho{n-1 \choose k-1}$. By the definition of $\r$, we have 
\[
({\lfloor}1/\rho{\rfloor}+1)\rho{n-1 \choose k-1} > \left(1+{\lfloor 1/\rho \rfloor+1 \choose 2}\r\right){n \choose k-1}.
\] 
By the Inclusion-Exclusion Principle, there exist $u, v \in V(H)$ such that $|N(u) \cap N(v)| \ge \r  {n \choose k-1}$, so they are $(K,\a,1)$-reachable by Lemma~\ref{LM4.2}.

By \eqref{eq:NK} and the above claim, we can apply Lemma \ref{lem:partition} on $H$ with the constants chosen at the beginning of the proof.
So we get a partition $\cP = \{V_1, \dots, V_r\}$ of $V(H)$ such that $r\le \min\{{\lfloor}1/\rho{\rfloor}, 1/(\rho-\r)\} = \lfloor 1/\rho \rfloor$ and for any $i\in [r]$, $|V_i|\ge (\rho-2\r)n$ and $V_i$ is $(K, \beta, 2^{c-1})$-closed in $H$. 

Now we show that $\bfu_i - \bfu_j \in L_{\cP, K}^{\mu}(H)$ for all distinct $i, j\in [r]$.
Without loss of generality, assume $i=1$ and $j=2$. % YZ revised below 
For any $u \in V_1$ and $v \in V_2$, since $\delta_{k-1}(H)\ge \rho n$, $u$ and $v$ are contained in at least ${{n-2} \choose {k-3}}\rho n/(k-2)$ edges. Since there are $\binom{k+r-1}{r-1}$ $k$-vectors, by averaging, there exists a $k$-vector $\bfv$ whose first two coordinates are positive and which is the index vector of at least
\[
\frac1{\binom{k}{2}}\cdot \frac{|V_1||V_2|{{n-2} \choose {k-3}}\rho n/(k-2)}{\binom{k+r-1}{r-1}} > \eta n^k
\]
edges (we divide a factor of $\binom k2$ because an edge may be counted at most $\binom k2$ times because of its intersections with $V_1$ and $V_2$).  Note that we can write $\bfv = \bfv_1+\bfv_2+\dots+\bfv_k $ such that all $\bfv_i\in \{\bfu_1,\dots, \bfu_r\}$. % YZ revised till here
Without loss of generality, assume that $\bfv_1=\bfu_1$ and $\bfv_2=\bfu_2$.

We apply Proposition \ref{supersaturation} with $a_1, a_2,\dots, a_k$ and conclude that there are at least $\mu n^m$ copies of $K$ in $H$ with index vector $\bfv'=a_1\bfu_1+a_2\bfu_2+a_3\bfv_3+\dots+a_k\bfv_k$.
We then apply Proposition \ref{supersaturation} again with $a_2, a_1,\dots, a_k$ (with $a_1, a_2$ exchanged) and conclude that there are at least $\mu n^m$ copies of $K$ in $H$ with index vector $\bfv''=a_2\bfu_1+a_1\bfu_2+ a_3\bfv_3 +\dots+a_k\bfv_k$.
By definition, we have $\bfv', \bfv''\in I_{\cP, K}^{\mu}(H)$ and thus $(a_1-a_2)\bfu_1+(a_2-a_1)\bfu_2 = \bfv' - \bfv'' \in L_{\cP, K}^{\mu}(H)$.
By repeating the arguments for other permutations of $a_1, a_2,\dots, a_k$, we get that $\bfw_i:=(a_{i+1}-a_i)\bfu_1+(a_i-a_{i+1})\bfu_2 \in L_{\cP, K}^{\mu}(H)$ for $i=1,2, \dots, k-1$.
Recall that since $\gcd(K)=1$, there exist $\ell_1, \ell_2, \dots, \ell_{k-1} \in \mathbb{Z}$ such that $\ell_1(a_2-a_1)+\ell_2(a_3-a_2)+\dots+\ell_{k-1}(a_k-a_{k-1})=1$. Hence $\bfu_1 - \bfu_2 = \ell_1\bfw_1 + \ell_2\bfw_2 + \dots + \ell_{k-1}\bfw_{k-1}\in L_{\cP, K}^{\mu}(H)$ and we are done.

Since $\bfu_i - \bfu_j \in L_{\cP, K}^{\mu}(H)$ for all distinct $i, j\in [r]$, by Lemma~\ref{lattice}, we conclude that $V(H)$ is $(K, \beta', i_0')$-closed.
Thus the desired absorbing set is provided by Lemma~\ref{LM1.1}.
\end{proof}

\section{Proof of the Almost Perfect Tiling Lemma}
% YZ changed k\ge 3 to k\ge 2 -- please check if I have done all in the section.
For integers $k\ge 2$ and $1\le a_1\le a_2\le \cdots \le a_k$ with $a_1<a_k$, let $K:=K^{(k)}(a_1,\dots, a_k)$. 
Let $m=|V(K)|=a_1+\cdots +a_k$ and $\sigma(K)=a_1/m$. 
Given an $n$-vertex $k$-graph $H$ such that $n$ is sufficiently large and $\delta_{k-1}(H)\ge (\sigma(K)-o(1)) n$, we will show that either $H$ has an almost perfect $K$-tiling, or $H$ is $\xi$-extremal for some $\xi>0$.
Note that it suffices to consider the case when $a_2=a_3=\cdots=a_k$. % YZ replaced equivalent by sufficient because the other direction is not clear.
In fact, let $K':=K^{(k)}((k-1)a_1, m-a_1, \dots, m-a_1)$ and note that $\sigma(K) = \sigma(K')=a_1/m$.
Moreover, since $K'$ has a perfect $K$-tiling, it suffices to find an almost perfect $K'$-tiling in $H$.

We thus consider $K' :=K^{(k)}(a, b,\dots, b)$ with $a<b$ and call the vertex class of size $a$ \emph{small}, and those of size $b$ \emph{large}. % i.e., $V(K')=A_1\cup \cdots \cup A_k$ with $|A_1|=a$ and $|A_2|=\cdots = |A_k|=b$.
Let $m=a+(k-1)b$.
An \emph{$\a$-deficient $K'$-tiling} of an $n$-vertex $k$-graph $H$ is a $K'$-tiling of $H$ that covers at least $(1-\a)n$ vertices. 
%We will prove the following (stronger) lemma, which implies Lemma~\ref{lemK}.
The following lemma allows a small number of $(k-1)$-subsets of $V(H)$ to have low degree, and may find applications in other problems (\eg, in reduced hypergraphs after we apply the regularity lemma).

\begin{lemma}\label{lemM}
Fix integers $k\ge 2$ and $a<b$, $0<\r\ll 1/m$ and let $K' :=K^{(k)}(a, b,\dots, b)$. For any $\a>0$ and $\xi \ge 5bk^2 \r$, there exist $\e>0$ and an integer $n_0$ such that the following holds. Suppose $H$ is a $k$-graph on $n>n_0$ vertices with $\deg(S)\ge (\sigma(K') - \gamma)n$ for all but at most $\e n^{k-1}$ sets $S\in \binom {V(H)}{k-1}$,
then $H$ has an $\a$-deficient $K'$-tiling or $H$ is $\xi$-extremal.
\end{lemma}

We will prove Lemma~\ref{lemM} by using the Weak Regularity Lemma for hypergraphs and the so-called fractional homomorphic tilings (introduced  by Bu\ss, H\`an and Schacht~\cite{BHS}). % YZ moved this sentence here

\subsection{Weak Regularity Lemma}

%Following the approach in \cite{HS}, we use the Weak Regularity Lemma, which is a straightforward extension of Szemer\'edi's regularity lemma for graphs \cite{Sze}.

Let $H = (V, E)$ be a $k$-graph and let $A_1, \dots, A_k$ be mutually disjoint non-empty subsets of $V$. We define $e(A_1, \dots, A_k)$ to be the number of edges with one vertex in each $A_i$, $i\in [k]$, and the density of $H$ with respect to ($A_1, \dots, A_k$) as
\[
d(A_1,\dots, A_k) = \frac{e(A_1, \dots, A_k)}{|A_1| \cdots|A_k|}.
\]
Given $\e>0$ and $d\ge 0$, we say a $k$-tuple ($V_1, \dots, V_k$) of mutually disjoint subsets $V_1, \dots, V_k\subseteq V$ is \emph{$(\e, d)$-regular} if
\[
|d(A_1, \dots, A_k) - d|\le \e
\]
for all $k$-tuples of subsets $A_i\subseteq V_i$, $i\in [k]$, satisfying $|A_i|\ge \e |V_i|$. We say ($V_1, \dots, V_k$) is \emph{$\e$-regular} if it is $(\e, d)$-regular for some $d\ge 0$. It is immediate from the definition that in an $(\e, d)$-regular $k$-tuple ($V_1, \dots, V_k$), if $V_i'\subset V_i$ has size $|V_i'| \ge c|V_i|$ for some $c\ge \e$, then ($V_1', \dots, V_k'$) is $(\e/c, d)$-regular.

The Weak Regularity Lemma (for hypergraphs) is a straightforward extension of Szemer\'edi's regularity lemma for graphs \cite{Sze}.

\begin{theorem}[Weak Regularity Lemma]
\label{thmReg}
Given $t_0\ge 0$ and $\e>0$, there exist $T_0 = T_0(t_0, \e)$ and $n_0 = n_0(t_0,\e)$ so that for every $k$-graph $H = (V, E)$ on $n>n_0$ vertices, there exists a partition $V = V_0 \cup V_1 \cup \cdots \cup V_t$ such that
\begin{enumerate}[(i)]
\item $t_0\le t\le T_0$,
\item $|V_1| = |V_2| = \cdots = |V_t|$ and $|V_0|\le \e n$,
\item for all but at most $\e \binom tk$ $k$-subsets $\{i_1,\dots, i_k\} \subset [t]$, the $k$-tuple $(V_{i_1}, \dots, V_{i_k})$ is $\e$-regular.
\end{enumerate}
\end{theorem}

The partition given in Theorem \ref{thmReg} is called an \emph{$\e$-regular partition} of $H$. Given an $\e$-regular partition $\Q$ of $H$ and $d\ge 0$, we refer to $ V_i, i\in [t]$ as \emph{clusters} and define the \emph{cluster hypergraph} $R = R(\e,d, \Q)$ with vertex set $[t]$ in which $\{i_1,\dots,i_k\}\subset [t]$ is an edge if and only if $(V_{i_1}, \dots, V_{i_k})$ is $\e$-regular and $d(V_{i_1}, \dots, V_{i_k}) \ge d$.

The following corollary shows that the cluster hypergraph inherits the minimum codegree of the original hypergraph.
The proof is standard and very similar to that of \cite[Proposition 16]{HS} so we omit the proof.
% YZ: do we require k\ge 3? Is $t_0\ge k\e^{-1/3}$ correct? Is e^{1/3} correct? \cite[Proposition 16]{HS} only requires e^{1/2}
\begin{corollary} \cite{HS}
\label{prop16}
%Given $ \e \ll c, d$ and integers $k\ge 2, t_0$ such that $t_0\sqrt\e \ge k$, there exist $T$ and $n_1$ such that the following holds. 
Suppose $1/t_0\ll \e \ll c, d, 1/k$, then there exist $T$ and $n_1$ such that the following holds. 
Let $H$ be a $k$-graph on $n>n_1$ vertices such that $\deg_H(S)\ge c n$ for all but at most $\e n^{k-1}$ $(k-1)$-sets $S$. Then $H$ has an $\e$-regular partition $\Q=\{V_0, V_1, \dots, V_t\}$ with $t_0\le t\le T$ such that in the cluster hypergraph $R = R(\e,d, \Q)$, all but at most $\sqrt \e t^{k-1}$ $(k-1)$-subsets $S$ of $[t]$ satisfy $\deg_{R}(S)\ge (c - d - 5\sqrt\e)t$.
\end{corollary}

\subsection{Fractional hom($K'$)-tiling}
Let $K' :=K^{(k)}(a, b,\dots, b)$ with $a<b$.
To obtain a large $K'$-tiling in the hypergraph $H$, we consider weighted homomorphisms from $K'$ into the cluster hypergraph $R$. For this purpose, we extend a definition that Bu\ss, H\`an and Schacht \cite{BHS} introduced for 3-graphs.\footnote{As noted by a referee, we could also use Farkas' lemma here (see~\cite{Chvatal}).}

\begin{definition}\label{frac}
Given a $k$-graph $H$, a function $h: V(H)\times E(H)\rightarrow [0,1]$ is called a fractional $\mathrm{hom}(K')$-tiling of $H$ if
\begin{enumerate}
\item [(1)] $h(v,e)= 0$ if $v\not\in e$,
\item [(2)] $h(v) = \sum_{e\in E(H)}h(v,e)\le 1$,
\item [(3)] for each $e\in E(H)$ there exists a labeling $e = v_1\cdots v_k$ such that
\[
h(v_1, e)\le \cdots \le h(v_k, e)\,\text{ and }\, \frac{h(v_1, e)}{a_1} \ge \cdots \ge \frac{h(v_k, e)}{a_k}
\]
\end{enumerate}
By $h_{\min}$ we denote the smallest non-zero value of $h(v,e)$ (and set $h_{\min} = \infty$ if $h\equiv 0$).
For $e=v_1\cdots v_k$, we write $h(v_1,\dots, v_k) = (h(v_1, e), \dots, h(v_k, e))$ and $h(e) = \sum_{i\in [k]} h(v_i,e)$.
The sum over all values is the weight $w(h)$ of h
\[
w(h) = \sum_{(v,e)\in V(H)\times E(H)}h(v,e).
\]
\end{definition}
%To ease notation, for $e=v_1\cdots v_k$, we write $h(v_1,\dots, v_k) = (h(v_1, e), \dots, h(v_k, e))$.

Assume that $V(K')=A_1\cup \cdots \cup A_k$ with $|A_1|=a$ and $|A_2|=\cdots = |A_k|=b$, and $m=a+(k-1)b$.

\begin{fact}\label{factFk}
There is a fractional $\mathrm{hom}(K')$-tiling of $K'$ such that $w(h) = m$ and $h_{\min}=b^{1-k}$.
\end{fact}

\begin{proof}
For each edge $e=v_1\cdots v_k$ of $K'$ with $v_i\in A_i$ for $i\in [k]$, we assign weight
\[
h(v_1,\dots, v_k) = \left(\frac1{b^{k-1}}, \frac{1}{a b^{k-2}}, \dots, \frac{1}{a b^{k-2}}\right).
\]
Then $h(e)= m/ (ab^{k-1})$ and consequently, $w(h)=m$.
\end{proof}

In the rest of the proof, we will refer to the weight assignment in Fact~\ref{factFk} as the \emph{standard weights}.

Let $\widehat{K'}$ be a $k$-graph obtained from $K'$ by adding $k-1$ new vertices and a new edge that consists of the $k-1$ new vertices and a vertex from a large vertex class of $K'$.
%on $A_1\cup \cdots \cup A_k$ with $|A_1|=a$ and $|A_i|=b$ for $2\le i\le k$ and $k-1$ vertices $u_1,\dots, u_{k-1}$ such that there is one edge $e=\{v, u_1,\dots, u_{k-1}\}\in E(L)$ such that $v\in A_2\cup \cdots \cup A_k$. 

\begin{proposition}\label{prop:frac}
There is a fractional $\mathrm{hom}(K')$-tiling of $\widehat{K'}$ such that $w(h)\ge m+1/(a b^{k-1})$ and $h_{\min}= b^{1-k}$.
\end{proposition}

\begin{proof}
Assume $V(K')= A_1\cup \cdots \cup A_k$ with $|A_1|=a$ and $|A_i|=b$ for $2\le i\le k$. 
Let $u_1,\dots, u_{k-1}$ be the vertices of $\widehat{K'}$ not in $V(K')$ and $e= \{v, u_1,\dots, u_{k-1}\}$ be the edge of $\widehat{K'}$ not in $E(K')$, where $v\in A_j$ for some $j>1$.
Fix any edge $e_1=x_1\cdots x_k$ of $K'$ with $x_i\in A_i$ for $i\in [k]$ such that $x_j=v$.
We assign the weight
\[
h(v, u_1,\dots, u_{k-1}) = \left(\frac1{ab^{k-2}}, \frac{1}{a^2 b^{k-3}}, \dots, \frac{1}{a^2 b^{k-3}}\right).
\]
to $e$ and the standard weight to all the edges of $K'$ except for $e_1$. 
Moreover, set $e_1$ as unweighted.
Since $\sum_{e'\in E(K')} h(e') = m - {m}/{(a b^{k-1})}$ and $h(e) = {m}/{(a^2 b^{k-2})}$ , we have $h_{\min}= b^{1-k}$ and
\[
w(h)=m - \frac{m}{a b^{k-1}} + \frac{m}{a^2 b^{k-2}} = m + \frac{1}{a b^{k-1}}\frac{m(b-a)}{a} \ge m + \frac{1}{a b^{k-1}}. \qedhere
\]
%and $h_{\min}= b^{1-k}$. 
\end{proof}

The following proposition says that a fractional hom($K'$)-tiling in the cluster hypergraph can be converted to an integer $K'$-tiling in the original hypergraph. It is almost the same as \cite[Proposition 4.4]{HZZ_tiling}, which covers the $k=3$ case, so we omit its proof.

\begin{proposition}\label{almost.frac}
Suppose $\e, \phi>0$, $d\ge 2\e/\phi$ and $t>0$ is an integer, and $n$ is sufficiently large.
Let $H$ be a $k$-graph on $n$ vertices with an $(\e, t)$-regular partition $\Q$ and a cluster hypergraph $\R=\R(\e, d,\Q)$. Suppose there is a fractional hom$(K')$-tiling $h$ of $\R$ with $h_{\min}\geq \phi$. Then there exists a $K'$-tiling of $H$ that covers at least $\left(1-2b\e/\phi \right) w(h) n/t $ vertices. \qed
\end{proposition}

%\begin{proof}
%Let $\R'$ be the subhypergraph of $\R$ consisting of the hyperedges $e=u_1\cdots u_k\in E(\R')$ with $h(u_1,e)$, \dots, $h(u_k,e)\geq h_{\min}$. For each $u\in V(\R')$, let $V_u$ be the corresponding cluster of $H$. Since $\Q$ is an $(\e, t)$-regular partition, all the clusters have size $N$ for some $N \ge (1 - \e) n/t$.
%In each $V_u$ we find disjoint subsets $V_u^e$ of size $h(u,e) N$ for all $e\in E(\R')$ with $u\in e$. Note that every edge $e=u_1\cdots u_k\in E(\R')$ corresponds to an $(\e, d')$-regular $k$-tuple $(V_{u_1}, \dots,V_{u_k})$ for some $d'\ge d$. Hence $(V_{u_1}^e, \dots,V_{u_k}^e)$ is $(b^{k-1}\e, d')$-regular.
%Because of Definition~\ref{frac} (3) and $d\ge 2b^{k-1}\e$, we can apply Proposition \ref{prop13} and obtain a $K'$-tiling covering at least
%\begin{align*}
%\left(1 - \frac{2b}{a}\cdot b^{k-1} \e \right) & h(e) N \ge (1- 2{b^k} \e) h(e) (1 - \e) \frac{n}{t} \ge \left(1 - 3 b^k\e \right)h(e) \frac{n}{t}
%\end{align*}
%vertices of $V_{u_1} \cup \cdots \cup V_{u_k}$, where $h(e)= \sum_{i=1}^k h(u_i,e)$.
%Repeating this to all hyperedges of $\R'$, we obtain a $K'$-tiling that covers at least
%\[
%\sum_{uvw\in E(\R')} \left(1- 3b^k \e \right) h(e) \frac{n}{t} = \left(1- 3b^k \e \right) w(h) \frac{n}{t}
%\]
%vertices of $H$.
%\end{proof}

% M tiling Lemma

\subsection{Proof of the $K'$-tiling Lemma (Lemma \ref{lemM})}

\begin{proposition}\label{propstab} 
For all $0< \rho \le 1/2$ and all $\xi, \beta,\delta,\e>0$ the following holds. Suppose there exists an $n_0$ such that for every $k$-graph $H$ on $n>n_0$ vertices satisfying $\deg(S)\ge \rho n$ for all but at most $\e n^{k-1}$ $(k-1)$-sets $S$, either $H$ has a $\beta$-deficient $K'$-tiling, or $H$ is $\xi$-extremal. 
Then every $k$-graph $H'$ on $n'>n_0$ vertices with $\deg(S)\ge (\rho -\delta) n'$ for all but at most $\e (n')^{k-1}$ $(k-1)$-sets $S$ has a $(\beta + 2\delta m)$-deficient $K'$-tiling or is $\xi$-extremal.
\end{proposition}

\begin{proof}
Given a $k$-graph $H'$ on $n' > n_0$ vertices with $\deg(S)\ge (\rho-\delta)n'$ for all but at most $\e (n')^{k-1}$ $(k-1)$-sets $S$. By adding a set $A$ of $2\delta n'$ new vertices to $H'$ and adding to $E(H)$ all $k$-subsets of $V(H')\cup A$ that intersect $A$, % YZ revised
we obtain a new hypergraph $H$ on $n = n' + 2\delta n'$ vertices. All $(k-1)$-subsets of $V(H)$ that intersect $A$ have degree $n - k+1$. All but at most $\e (n')^{k-1}$ $(k-1)$-subsets $S$ of $V(H')$ satisfy 
\[
\deg(S)\ge (\rho-\delta)n' + 2\delta n' = \rho n' + \delta n' = \rho n - 2 \rho \delta n' + \delta n' \ge \rho n
\]
because $\rho\le 1/2$. By assumption, either $H$ has a $\beta$-deficient $K'$-tiling, or $H$ is $\xi$-extremal. 
If $H$ has a $\beta$-deficient $K'$-tiling, then by removing all the $K'$-copies that intersect $A$, we obtain a ($\beta + 2\delta m$)-deficient $K'$-tiling of $H'$. Otherwise $H$ is $\xi$-extremal, namely, there exists a set $B\subseteq V(H)$ of $\lfloor(1- \sigma(K'))n\rfloor$ vertices such that $e(B)\le \xi \binom{|B|}{k}$.
We can assume that $A\cap B=\emptyset$ -- otherwise we swap the vertices in $A\cap B$ with the vertices in $V\setminus (A\cup B)$ and $e(B)$ will not increase. %Then we remove $A$ from $H$ and get a sparse set $B$ in $H'$ of order at least $\lfloor(1-\sigma(K'))n\rfloor > \lfloor(1-\sigma(K'))n'\rfloor $. 
By averaging, there exists a subset $B'\subseteq B\subseteq V(H')$ of order exactly $\lfloor(1-\sigma(K'))n'\rfloor$ such that $e(B')\le \xi \binom{|B'|}{k}$.
Thus, $H'$ is $\xi$-extremal.
\end{proof}

Now we are ready to prove Lemma \ref{lemM}.
% proof of M tiling lemma
\begin{proof}[Proof of Lemma \ref{lemM}]
Fix positive integers $a< b$ and $k\ge 2$, and a real number $0<\r \ll 1/m$.  % YZ we don't need k\ge 2 but do need a<b
Let $m:= a + (k-1)b$ and $\sigma:= \sigma(K')= a/m$.
Trivially the lemma works when $\a\ge 1$ or $\xi\ge 1$.
Assume to the contrary, that there exist $\a\in (0,1), \xi\in [5b k^2 \r,1)$ such that for all $\e_0>0$ and integers $n_0$, Lemma \ref{lemM} fails, namely, there is a $k$-graph $H$ on $n> n_0$ vertices which satisfies $\deg(S)\ge (\sigma- \r)n$ for all but at most $\e_0 n^{k-1}$ $(k-1)$-sets $S$ but which does not contain an $\a$-deficient $K'$-tiling and is not $\xi$-extremal.
Let $\Gamma$ be the set of such pairs $(\a, \xi)$. % YZ removed such that Lemma~\ref{lemM} fails.
Define $f(\a, \xi)=\a + \r \a^2 \xi$, and let $f_0$ be the supremum of $f(\a, \xi)$ among all $(a, \xi)\in \Gamma$.

Let $\eta = \r^2 f_0^2/32$.
By the definition of $f_0$, there exists a pair $(\a_0, \xi_0)\in \Gamma$ such that $f_0 - \eta \le f(\a_0, \xi_0)\le f_0$. 
Moreover, since $1+\r \a_0 \xi_0 \le 2$, we have that $\a_0 \ge (f_0 - \eta)/2 \ge f_0/4$.
%Let $d=\min\{(\r \a_0)^{2}/4m, (\xi_0 - 5k m\r)/2\}$ and thus $\xi_0 - 2d \ge 5k m\r$. YZ changed the definition of d.
Let $d= (\r \a_0)^2/(4m)$.
Since $f(\a_0, \xi_0) \ge f_0 - \eta$, we have
\[
f(\a_0 + (\r \a_0)^{2}, \xi_0- 5d) > \a_0 + (\r \a_0)^{2} + \r \a_0^2 (\xi_0 - 5d) \ge f_0 -\eta + \r \a_0^2(\r - 5d) \ge f_0
\]
by $d\le \r^2 \le \r/10$ and $\eta = \r^2 f_0^2/32 \le \r^2 \a_0^2/2$.
This means that $(\a_0 + (\r \a_0)^{2}, \xi_0- 5d)\notin \Gamma$, i.e., there exist $\e_*>0$ and $n_*\in \mathbb N$ such that for every $k$-graph $H$ on $n>n_*$ vertices satisfying $\deg(S)\ge (\sigma- \r)n$ for all but at most $\e_* n^{k-1}$ $(k-1)$-sets $S$, $H$ has an $(\a_0 + (\r \a_0)^{2})$-deficient $K'$-tiling or is $(\xi_0 - 5d)$-extremal. 
Since $\sigma < 1/k \le 1/2$, we can apply Proposition \ref{propstab} with parameters $\delta=2d$ and $\sqrt\e$ and derive that
\begin{itemize}
\item [($\dagger$)]
for every $k$-graph $H'$ on $n>n_*$ vertices satisfying $\deg(S)\ge (\sigma- \r - 2d)n$ for all but at most $\e_* n^{k-1}$ $(k-1)$-sets $S$, either $H'$ has an $(\a_0 + (\r \a_0)^{2} + 4d m)$-deficient $K'$-tiling, or it is $(\xi_0 - 5d)$-extremal. 
\end{itemize}
Let $\e>0$ be such that $\e\le \min\{\e_*^2, d^k, (1/k-\sigma)/3, d/(2b^k)\}$ and $\e\ll c, d, 1/k$ as required by Corollary~\ref{prop16}. 
Let $n_1$ and $T$ be the parameters returned by Corollary \ref{prop16} with inputs $c=\sigma- \r$, $\e$, $d$, $k$ and sufficiently large $t_0$ (in particular, $t_0> n_*$). %Let $n\ge n_1$ be sufficiently large. % YZ replaced $t_{0} = \max \{2n_0, 1/\e\}$ because 1) why 2n_0 2) we need t (and |U|) be sufficiently large, 1/\e may not suffice.

Let $H$ be a $k$-graph on $n\ge n_1$ vertices which satisfies $\deg(S)\ge (\sigma- \r)n$ for all but at most $\e n^{k-1}$ $(k-1)$-sets $S$.
Our goal is to show that either $H$ contains an $\a_0$-deficient $K'$-tiling or $H$ is $\xi_0$-extremal.
This implies that $(\a_0, \xi_0)\notin \Gamma$, contradicting the definition of $(\a_0, \xi_0)$.
Let us apply Corollary \ref{prop16} to $H$ with the constants chosen above and obtain a cluster hypergraph $R = R(\e, d, \Q)$ on $t\ge t_{0}$ vertices such that the number of $(k-1)$-subsets $S\subseteq V(R)$ violating 
\[
\deg_{R}(S)\ge \left(\sigma- \r - d - 5\sqrt\e \right)t \ge \left(\sigma- \r - 2d \right)t
\]
is at most $\sqrt{\e} t^{k-1}\le \e_* t^{k-1}$. 
Let $N$ be the number of vertices in each cluster except $V_0$ and thus $(1-\e)\frac nt \le N \le \frac nt$.

Note that the reduced $k$-graph $R$ satisfies the assumption of ($\dagger$). 
If $R$ is $(\xi_0 - 5d)$-extremal, then there exists a vertex set $B\subseteq V(R)$ of order $\lfloor (1-\sigma)t \rfloor$ such that $e(B)\le (\xi_0 - 5d) \binom{|B|}{k}$. 
Let $B'\subseteq V(H)$ be the union of the clusters in $B$. By regularity, we have
\[
e(B')\le e(B)\cdot N^k + \binom tk \cdot d \cdot N^k + \e \binom tk \cdot N^k + t\binom {N}2 \binom {n-2}{k-2},
\]
where the right-hand side bounds the number of edges from regular $k$-tuples with high density, edges from regular $k$-tuples with low density, edges from irregular $k$-tuples and edges that lie in at most $k-1$ clusters. 
Since $e(B)\le (\xi_0 - 5d) \binom{|B|}{k}$ and $1/t \le \e$, we get
\begin{align*}
e(B')&\le (\xi_0 - 5d) \binom{|B|}{k} N^k + (d+ \e) \binom tk \left(\frac nt \right)^k + t\binom {n/t}2 \binom {n-2}{k-2}\\ 
&\le (\xi_0 - 5d) \binom{|B'|}{k} + (d+ 2\e) \binom{n}{k}.
\end{align*}
Note that
\[
|B'|= \lfloor (1-\sigma)t \rfloor N \ge (1-\sigma)(1- \e)n - N\ge (1- \sigma - 2\e)n.
\]
On the other hand, $|B'|\le (1-\sigma) n$ implies that $|B'|\le \lfloor (1-\sigma)n \rfloor$. 
By adding at most $2\e n$ vertices of $V\setminus B'$ to $B'$, we obtain a subset $B''\subseteq V(H)$ of size exactly $\lfloor (1-\sigma)n \rfloor$, with $e(B'')\le e(B') + 2\e n \binom{n-1}{k-1}\le e(B') + 2k\e \binom nk$. 
Since  $\sigma< 1/k$ and $\e\le (1/k-\sigma)/3$, we have
\begin{align*}
\binom{|B'|}k &\ge \binom{(1- \sigma - 2\e)n}k\ge (1-\sigma - 2\e)^k \binom nk - O(n^{k-1}) \ge \left(1 - \frac1k \right)^k\binom nk.
\end{align*}
Since $(1- 1/k)^k\ge 1/4$ for $k\ge 2$, it follows that $\binom{|B'|}k \ge \frac14 \binom nk$, and consequently, 
\[
e(B'')\le  (\xi_0 - 5d) \binom{|B'|}{k} + (d+ 2\e + 2k\e) \binom{n}{k} \le  \Big(\xi_0 - 5d + 4(d+2\e+ 2k\e)\Big) \binom{|B'|}{k}  \le \xi_0 \binom{|B'|}{k}
\]
by $\e\le d^k$. Hence $H$ is $\xi_0$-extremal, and we are done.

\medskip
We thus assume that $R$ is not $(\xi_0 - 5d)$-extremal. By ($\dagger$), $R$ has an $(\a_0 + (\r \a_0)^{2} + 4d m)$-deficient $K'$-tiling.
Let $\mathcal{K}$ be a largest $K'$-tiling of $R$, and let $U$ be the set of vertices in $R$ not covered by $\mathcal{K}$. Then $|U|\le (\a_0 + (\r \a_0)^{2} + 4d m)t$. 
Let $q:=|\K|$.

Now assume that $H$ contains no $\a_0$-deficient $K'$-tiling. % YZ: this was not stated clearly before.
The following proposition shows that there is no fractional hom$(K')$-tiling of $R$ whose weight is substantially larger than $(1- \a_0)t$.

\begin{claim}\label{clm:gain}
If $h$ is a fractional hom$(K')$-tiling of $R$ with $h_{\min}\geq {b^{1-k}}$, then $w(h)\le (1 - \a_0 + 2b^k\e)t\le q m + 3(\r \a_0)^{2} t$.
\end{claim}

\begin{proof}
We first observe that $H$ contains an $\a_0$-deficient $K'$-tiling if there is a fractional hom$(K')$-tiling $h$ of $R$ with $h_{\min}\geq {b^{1-k}}$ and $w(h)\ge (1 - \a_0 + 2b^k\e)t$. This indeed follows from Proposition \ref{almost.frac} because $2b^k \e \le d = (\r \a_0)^2/ (4m) < \a_0$ and
\begin{align*}
\left(1- \frac{2b\e}{b^{1-k}} \right)w(h) \frac nt \geq \left(1-2b^k\e\right) (1 - \a_0 + 2b^k \e)t \frac nt > \left(1-\a_0\right)n.
\end{align*}

It remains to show that $q m + 3(\r \a_0)^{2} t \ge (1 - \a_0 + 2b^k\e)t$. 
Since $|U|\le (\a_0 + (\r \a_0)^{2} + 4d m)t$ and $|U| + qm = t$, we have
\[
q m+3(\r \a_0)^{2}t \ge t - (\a_0 + (\r \a_0)^{2} + 4dm) t+ 3(\r \a_0)^{2} t \ge (1-\a_0+ 2b^k\e)t
\]
because $4dm = (\r \a_0)^2$ and $2b^k \e < (\r \a_0)^2$.
\end{proof}

We claim the following for $|U|$:
\begin{equation}\label{eqU}
|U| \ge \frac{\a_0}{2} t \quad \text{and} \quad  e(U)\le \frac{\r}k \binom{|U|}k. % YZ removed the upper bound |U| \le (1-\sigma + 2\r) t because it is not used. Also add the bound for e(U) here.
\end{equation}
Indeed, if $|U|<{\a_0 t}/{2}$, then by applying Fact \ref{factFk} to each member of $\mathcal{K}$, we obtain a fractional $\mathrm{hom}(K')$-tiling $h$ of $R$ with $h_{\min}=b^{1-k}$ and weight $w(h)=(1-\a_0 /2)t > (1-\a_0+2b^k\e)t$. 
This contradicts Claim~\ref{clm:gain}.  
If $e(U)\ge \r \binom{|U|}k/k$, then since $|U| \ge {\a_0 t}/{2}$ and $t$ is sufficiently large, we can apply \eqref{eq:erdos} to find a copy of $K'$ in $U$, contradicting the maximality of $\mathcal K$.

%On the other hand, if $|U| > (1-\sigma +2 \r) t$, by the degree condition, we know that every $(k-1)$-set in $\binom{U}{k-1}$ with degree at least $(\sigma- \r - 2d)t$ has at least $(\r-2d)t$ neighbors in $U$.
%This implies that 
%\[
%e(U)\ge \frac{(\r-2d) t}k \left(\binom{|U|}{k-1}- \sqrt\e t^{k-1}\right)>\frac{\r}{2k} |U| \left(\binom{|U|}{k-1}- \r\binom{|U|}{k-1}\right)>\frac{\r}{k}\binom{|U|}{k},
%\]
%where we used \eqref{eq:sqrte}, $k\ge 3$ and $\r\ll 1/k$.
%By~\eqref{eq:erdos} and the fact that $t$ is large enough, we find a copy of $K'$ in $U$, contradicting the maximality of $\mathcal K$.
%So the proof of \eqref{eqU} is complete. 
%Note that the argument above also implies that $e(U)\le \r \binom{|U|}k/k$.

Let $D$ be the set of vertices $v\in V(R)\setminus U$ such that $|N(v)\cap \binom{U}{k-1}|\ge \r \binom{|U|}{k-1}$. Let $\mathcal K_1\subseteq \mathcal K$ be the set of copies of $K'$ that contain at least $a+1$ vertices from $D$. Let $\mathcal K_2\subseteq \mathcal K$ be the set of copies of $K'$ that contain exactly $a$ vertices from $D$. Let $\mathcal K_3\subseteq \mathcal K$ be the set of copies of $K'$ that contain at most $a-1$ vertices from $D$. 

% YZ: the original bound \r |U| / k^2 was correct but may not be smaller than \r t/m (when b is large). So I have to reduce the bound. In addition, I don't find a non-trivial upper bound for |U| useful. 
\begin{claim}\label{clm:k-1}
$|\mathcal K_1|< \frac{\r}{m(k-1)^2} |U| < \frac{\r}{m(k-1)^2} t $. 
\end{claim}

\begin{proof}
Suppose $|\mathcal K_1|\ge \frac{\r}{m(k-1)^2} |U|$ instead. Let $K_{1}, \dots K_{{\ell}}$ be distinct members of $\mathcal K_1$, where $\ell= \lceil \frac{\r |U|}{m (k-1)^2} \rceil$. 
For each $j\le \ell$, since $|V(K_j)\cap D|\ge a+1$, we can find a vertex $v_j\in D$ from a large vertex class of $K_j$. Since 
\[
\r \binom{|U|}{k-1} > \frac{\r |U|}{m(k-1)} \binom{|U|-1}{k-2} > (\ell - 1) (k-1)\binom{|U|-1}{k-2}, 
\]
we can greedily pick disjoint $(k-1)$-sets $S_{1}, \dots, S_{\ell}$ from $U$ such that $S_j \in N(v_j)$. 
%$R[V(K_{i_j})\cup U_j]$ contains $L$ as a subhypergraph, which is because 
Thus, for each $j\le \ell$, by Proposition \ref{prop:frac}, we get a fractional $\mathrm{hom}(K')$-tiling $h^j$ of $R[V(K_{j})\cup S_j]$ such that $w(h^j)\ge m+1/(a b^{k-1})$ and $h_{\min}^j\ge b^{1-k}$. 
We assign the standard weight to other members of $\mathcal K$ and thus obtain a fractional $\mathrm{hom}(K')$-tiling $h$ of $R$ with $h_{\min}\ge b^{1-k}$ and weight
\begin{align*}
w(h) \ge q m + \frac{\r |U|}{m(k-1)^2}\frac1{a b^{k-1}} \ge q m + \frac{\r \a_0 t}{2m(k-1)^2 a b^{k-1}} > q m + 3(\r \a_0)^{2} t,
\end{align*}
where we used \eqref{eqU} and that $\r$ is small.
This contradicts Claim~\ref{clm:gain}.
\end{proof}

We now find an upper bound for $|\mathcal K_3|$. %Now consider $\sum_{S\in \binom{U}{k-1}}\deg(S)$. 
First, by the definitions of $\mathcal K_1, \mathcal K_2, \mathcal K_3$ and $e(U)\le \r \binom{|U|}k/k$, we have
\begin{align*}
\sum_{S\in \binom{U}{k-1}}\deg(S) &\le (m|\mathcal K_1|+a|\mathcal K_2|+(a-1)|\mathcal K_3|)\binom{|U|}{k-1}+q m\r \binom{|U|}{k-1}+\r \binom{|U|}k \\
&\le a q \binom{|U|}{k-1} + ((m-a)|\mathcal K_1| - |\mathcal K_3|)\binom{|U|}{k-1}+ \r \binom{|U|}{k-1}t,
\end{align*}
where the second inequality follows from $|\mathcal K_1|+|\mathcal K_2|+|\mathcal K_3|=q$ and $q m+|U|=t$.
Second, the degree condition of $R$ implies that
\begin{align*}
\sum_{S\in \binom{U}{k-1}}\deg(S) &\ge \left( \binom{|U|}{k-1} - \sqrt\e t^{k-1}  \right) \left(\sigma- \r - 2d \right)t. 
\end{align*}
% YZ move the estimate on \sqrt{\e} t^{k-1} here.
Since $|U|\ge {\a_0 t}/{2}$ and ${\e}\le d^k\le (\r \a_0)^{2k}$, we have
\begin{equation*}
\sqrt{\e} t^{k-1} \le {(\r \a_0)^{k}} \left( \frac{2|U|}{\a_0} \right)^{k-1} \le \r \binom{|U|}{k-1}
\end{equation*}
as $\r \le1/(2k)$ and $|U|$ is sufficiently large. It follows that
\begin{align*}
\sum_{S\in \binom{U}{k-1}}\deg(S) \ge \left( \binom{|U|}{k-1} - {\r}\binom{|U|}{k-1} \right) \left( \frac am - (\r + 2d) \right)t \ge \binom{|U|}{k-1} a q - 2\r \binom{|U|}{k-1}t,
\end{align*}
where the last inequality holds because $2d \le \r (m-a)/m$ and $q m\le t$.
Comparing the upper and lower bounds for $\sum_{S\in \binom{U}{k-1}}\deg(S)$ gives
\[
|\mathcal K_3|\le (m-a)|\mathcal K_1| +3\r t.
\]
By Claim~\ref{clm:k-1}, it follows that
\begin{equation}\label{eq:F13}
|\mathcal K_1|+|\mathcal K_3|\le |\mathcal K_1|+(m-a)|\mathcal K_1| +3\r t \le 4\r t.
\end{equation}
Thus we have $|\mathcal K_2|\ge q - 4\r t$.

% YZ remove the defintion of D' Let $D'\subseteq D$ be the set of vertices that lie in copies of $K'$ belonging to $\mathcal K_2$. 
%Let $A=(V(\mathcal K_2)\setminus D')\cup U$.
Let $A$ be the union of $U$ and the vertices covered by $\mathcal K_2$ but not in $D$.
Then %$|A|\ge (m-a)(q-4\r t)$. %Recall that $q m + |U|=t$, so we have
\begin{align}\label{eq:A} % YZ, we now need a lower bound for |A| because of the new way of defining extremal.
|A|&\ge |U| + (m-a)(q-4\r t)  %\ge (m-a)(q-4\r t) + |U| (m-a)/m 
\ge (1-\sigma) t - 4(m-a) \r t \ge \left(1 - \frac1k \right)t
\end{align}
because $\r$ is small. % here need $\r \le \frac{k-1}{4km^2}$. 
We claim that we can find a set $\M$ of $\ell'$ disjoint edges in $A$, where $\ell'=\lceil \frac{\r |U|}{k m^2}\rceil$.
Indeed, by deleting some vertices or adding at most $4(m-a) \r t$ vertices from $V(R)\setminus A$ to $A$, we can obtain a set $A'$ of size exactly $\lfloor(1-\sigma)t\rfloor$.
Since $R$ is not $(\xi_0 - 5d)$-extremal, we have that $e(A')\ge (\xi_0 - 5d) \binom{|A'|}{k} $. 
%Now define $A_0 = A'$ if $|A|>\lfloor(1-\sigma)t\rfloor$ and $A_0 = A$ otherwise.
If $|A|>\lfloor(1-\sigma)t\rfloor$, then as $e(A') \ge (\xi_0 - 5d) \binom{|A'|}{k} > (\ell'-1) k \binom{|A'|-1}{k-1}$, we can find the desierd $\M$ in $A'\subseteq A$.
Otherwise, for $|A|\le \lfloor(1-\sigma)t\rfloor$, by \eqref{eq:A}, it follows that
\begin{align*}
e(A) &\ge e(A') - 4(m-a)\r t\cdot \binom{|A|-1}{k-1}\ge  
(\xi_0 -5d) \binom{|A'|}{k} - 4(m-a)\r \frac{|A|}{1- 1/k}\binom{|A|-1}{k-1}\\
&\ge (\xi_0 - 5d - 4b k^2 \r) \binom{|A|}{k} > \r  \binom{|A|}{k} > (\ell'-1) k \binom{|A|-1}{k-1}
\end{align*}
because $\xi_0\ge 5b k^2\r$ and $5d< \r$. % YZ this is where the definition of \xi_0 comes from. The old definition of d is unnecessary 
%Let $\ell'=\lceil \frac{\r |U|}{k m^2}\rceil$. 
%Since $e(A)> \r \binom{|A|}{k} > (\ell'-1) k \binom{|A|-1}{k-1}$. 
Thus we can greedily find the desired $\M$ in $A$. 

Let $K_{1},\dots, K_{p}$  denote the members of $\mathcal K_2$ that intersect the edges of $\M$, where $p\le k \ell'$. 
%(they may also intersect $U$). 
For each $j\in [p]$, suppose $V(K_{j})\cap D =\{v_{j,1},\dots v_{j, a}\}$. For each $j\in [p]$ and $i\in [a]$, we claim that we can greedily find disjoint copies $K_{j, i}'$ of complete $(k-1)$-partite $(k-1)$-graphs $K^{(k-1)}(b, \dots, b)$ in $N(v_{j,i})\cap \binom{U}{k-1}$. Indeed, during the process, at most $p a (k-1)b+ k \ell'$ vertices of $U$ are occupied and the number of $(k-1)$-subsets of $U$ intersecting these vertices is at most
\[
(p a (k-1)b+ k \ell')\binom{|U| -1}{k-2}\le a m k \ell'\binom{|U| -1}{k-2} \le \frac{k-1}k \r  \binom{|U|}{k-1}
\]
because $\ell'\le \r |U|/(k m^2) +1$, $ak< m$, and $|U|$ is sufficiently large. Thus, since
\[
\left|N(v_{j,i})\cap \binom{U}{k-1}\right| - \frac{k-1}{k}\r \binom{|U|}{k-1} \ge \r\binom{|U|}{k-1} -  \frac{k-1}{k}\r  \binom{|U|}{k-1}= \frac{\r}k \binom{|U|}{k-1},
\] 
we can apply \eqref{eq:erdos} to find the desired $K_{j, i}'$ for all $v_{j, i}$.

For each $j\in [p]$ and $i\in [a]$, $R[V(K_{j, i}')\cup\{v_{j, i}\}]$ spans a copy of $K^{(k)}(1, b, \dots, b)$. 
We now assign the weight $(1/b^{k-1}, 1/(a b^{k-2}), \dots, 1/(ab^{k-2}))$ to each edge of this $K^{(k)}(1, b, \dots, b)$ such that the weight of $v_{j, i}$ is $1/b^{k-1}$. %and the weight of other vertices in $ 1/(a b^{k-2})$
The total weight of $R[V(K_{j, i}')\cup\{v_{j, i}\}]$ is thus $1 + b(k-1)/a = m/a$. 
Next we assign the standard weight to each member of $\mathcal K\setminus \{K_{1},\dots, K_{p}\}$. 
Finally, we assign weight $(1,\dots, 1)$ to all the edges of $\M$. This gives a fractional hom($K'$)-tiling $h$ with $h_{\min}=b^{1-k}$ and weight
\[
w(h) = p a\frac{m}{a} + (q-p)m + k \ell' \ge q m + \frac{\r |U|}{m^2} \ge q m + \frac{\r \a_0 t}{2m^2} > q m + 3(\r \a_0)^2 t
\]
where we used \eqref{eqU} and that the assumption $\r$ is small.
This contradicts Claim~\ref{clm:gain} and completes our proof.
\end{proof}

\section{The Extremal Case}

In this section we prove Theorems~\ref{lemE},~\ref{lemE3} and~\ref{lemE2}. 
We first give some notation.
Given two disjoint vertex sets $X$ and $Y$ and two integers $i, j\ge 0$, a set $S\subset X\cup Y$ is called an $X^i Y^j$-set if $|S\cap X|= i$ and $|S\cap Y| =j$. When $X, Y$ are two disjoint subsets of $V(H)$ and $i+j=k$, we denote by $H(X^i Y^j)$ the family of all edges of $H$ that are $X^i Y^j$-sets, and let $e_{H}(X^i Y^{j})= |H(X^i Y^j)|$ (the subscript may be omitted if it is clear from the context).
We use $\overline{e}_{H}(X^i Y^{k-i})$ to denote the number of non-edges among $X^i Y^{k-i}$-sets. Given a set $L\subseteq X\cup Y$ with $|L\cap X|=l_1\le i$ and $|L\cap Y|=l_2\le k-i$, we define $ \deg(L, X^i Y^{k-i})$ as the number of edges in $H(X^i Y^{k-i})$ that contain $L$, and
$\overline \deg(L, X^i Y^{k-i})=\binom{|X|-l_1}{i-l_1} \binom{|Y|-l_2}{k-i-l_2} - \deg(L, X^i Y^{k-i})$. Our earlier notation $\deg(S, R)$ may be viewed as $\deg(S, S^{|S|} (R\setminus S)^{k- |S|})$.

\subsection{Tools and a general setup}
The following lemma deals with a special (ideal) case of Theorem~\ref{lemE}.
We postpone its proof to the end of this section.

\begin{lemma}\label{lem3}
Let $K:=K^{(k)}(a_1, \dots, a_k)$ with $m:=a_1+\cdots+a_k$.
Suppose  $1/n\ll \rho \ll 1/m$ and $n\in m\mathbb{Z}$. 
Suppose $H$ is a $k$-graph on $n$ vertices with a partition of $V(H)=X\cup Y$ such that $a_1|Y|= (m-a_1)|X|$. Furthermore, assume that
\begin{itemize}
\item for every vertex $v\in X$, $\overline \deg(v, Y)\le \rho\binom{|Y|}{k-1}$,
\item for every vertex $u\in Y$, $\overline{\deg}(u, XY^{k-1})\le \rho \binom{|Y|}{k-1}$. 
\end{itemize}
Then $H$ contains a $K$-factor.
\end{lemma}

We use the following simple fact in the proof.

%\begin{fact}\label{lemEX}
%%Let $t, m\in \mathbb{N}$, $\e>0$ and let $n$ be a sufficiently large integer.
%Let $t, m, n\in \mathbb{N}$ and let $\e>0$.
%Let $F$ be an $m$-vertex $k$-partite $k$-graph.
%Let $H$ be an $n$-vertex $k$-graph with maximum vertex degree $\e \binom{n}{k-1}$ and $e(H)> (t-1)m \e \binom{n}{k-1} + {\rm ex}(n', F)$, where $n'=n-(t-1)m$.
%Then $H$ contains an $F$-tiling of size $t$.
%\end{fact}
%
%\begin{proof}
%Assume to the contrary, that the largest $F$-tiling $M$ in $H$ has size at most $t-1$.
%Let $V'$ be a set of $(t-1)m$ vertices containing $V(M)$.
%Thus, $V(H)\setminus V'$ spans no copy of $F$ and thus spans at most ${\rm ex}(n', F)$ edges.
%So we have $e(H)\le (t-1)m\e \binom{n}{k-1} + {\rm ex}(n', F)$, a contradiction.
%\end{proof}
\begin{fact}\label{lemEX}
%Let $t, m\in \mathbb{N}$, $\e>0$ and let $n$ be a sufficiently large integer.
Let $t, m, n\in \mathbb{N}$ and $F$ be an $m$-vertex $k$-partite $k$-graph.
Let $H$ be an $n$-vertex $k$-graph with maximum vertex degree $\Delta$ and $e(H)> (t-1)m \Delta + {\rm ex}(n', F)$, where $n'=n-(t-1)m$.
Then $H$ contains an $F$-tiling of size $t$.
\end{fact}

\begin{proof}
Assume to the contrary, that the largest $F$-tiling $M$ in $H$ has size at most $t-1$.
Let $V'$ be a set of $(t-1)m$ vertices containing $V(M)$.
Thus, $V(H)\setminus V'$ spans no copy of $F$ and thus spans at most ${\rm ex}(n', F)$ edges.
So we have $e(H)\le (t-1)m \Delta + {\rm ex}(n', F)$, a contradiction.
\end{proof}

Now we start with a general setup and prove some estimates. 
Assume that $k\ge 3$, $a_1\le \cdots\le a_k$ and $m=a_1+\cdots +a_k$.
Suppose $1/n \ll \xi\ll 1/m$ such that $n\in m \mathbb{N}$. Let $H$ be a $k$-graph on $V$ of $n$ vertices such that $\delta_{k-1}(H) \ge \frac{a_1}{m}n$. Furthermore, assume that there is a set $B\subseteq V(H)$, such that $|B| = \frac{m-a_1}{m}n$ and $e(B)\le \xi \binom{|B|}{k}$. Let $A=V\setminus B$. Then $|A|= \frac{a_1}{m}n$.
%For the convenience of later calculations, we let $\e_0 =2 (\frac{m}{m-a_1})^k\xi \ll 1$ and thus
%\begin{equation}\label{eqB}
%e(B)\le \xi \binom{n}{k} = \frac{\e_0}2 \left(1-\frac{a_1}{m} \right)^{k} \binom{n}{k} \le \e_0 \binom {|B|}k.
%\end{equation}
%In addition, assume that $e(B)$ is the smallest among all partitions $V(H)=A\cup B$ such that $|B| = \frac{m-a_1}{m}n$ and $e(B)\le \xi \binom {|B|}k$. 
Note that we only require $\delta_{k-1}(H)\ge \frac{a_1}{m}n$, so that we can use the estimates in all proofs.

Let $\e_1={\xi}^{1/7} $, $\e_2=\xi^{1/3}$ and $\e_3=2\xi^{2/3}$. 
We now define
\begin{align*}
&A':=\left\{ v\in V: \overline{\deg} (v,B)\le \e_2 \binom{|B|}{k-1} \right\}, \\
&B':=\left\{ v\in V: \deg (v,B)\le \e_1\binom{|B|}{k-1} \right\},\, V_0:=V\setminus(A'\cup B').
\end{align*}

\begin{claim}\label{clm:size}
$\{|A\setminus A'|, |B\setminus  B'|, |A'\setminus  A|, |B'\setminus  B|\}\le \e_3|B|$ and $|V_0|\le 2\e_3|B|$.
\end{claim}

\begin{proof}
First assume that $|B\setminus B'|> \e_3|B|$. By the definition of $B'$, we get that
\[
e(B) > \frac 1k \e_1\binom {|B|}{k-1} \cdot \e_3 |B| > 2\xi\binom {|B|}k,
\]
which contradicts $e(B)\le \xi \binom{|B|}{k}$.

Second, assume that $|A\setminus A'|> \e_3 |B|$. Then by the definition of $A'$, for any vertex $v\notin A'$, we have that $\overline \deg(v,B)> \e_2\binom{|B|}{k-1}$. So we get
\[
\overline e(AB^{k-1})> \e_3 |B| \cdot \e_2 \binom{|B|}{k-1} = 2\xi|B|  \binom{|B|}{k-1}.
\]
Together with $e(B)\le \xi \binom{|B|}{k}$, this implies that
\begin{align*}
\sum_{S\in \binom B{k-1}}\overline\deg(S)&= k\overline{e}(B) + \overline{e}(AB^{k-1}) \\
                              &> k(1 - \xi) \binom{|B|}k + 2\xi|B| \binom{|B|}{k-1}  \\
                              &= ((1-\xi)(|B|-k+1)+2\xi|B|) \binom{|B|}{k-1} > |B|  \binom{|B|}{k-1}.
\end{align*}
where the last inequality holds because $n$ is large enough. By the pigeonhole principle, there exists a set $S\in \binom{B}{k-1}$, such that
$\overline\deg(S) > |B|= \frac{m-a_1}{m}n$, contradicting $\delta_{k-1}(H) \ge \frac{a_1}{m}n$.

Consequently,
\begin{align*}
&|A'\setminus A|=|A'\cap B|\le |B\setminus B'|\le \e_3|B|,   \\
&|B'\setminus B|=|A\cap B'|\le |A\setminus A'|\le \e_3|B|, \\
&|V_0|\le |A\setminus A'|+|B\setminus B'|\le \e_3|B|+\e_3|B|=2\e_3 |B|. \qedhere
\end{align*}
\end{proof}

Note that by $|B\setminus  B'|\le \e_3 |B|$, we have
\begin{equation}\label{eq:V_0}
\deg(w, B')\ge \deg\left(w,B \right)-|B\setminus B'| \binom{|B|}{k-2} \ge \frac{\e_1}2\binom{|B'|}{k-1} \text{ for any vertex }w\in V_0,
\end{equation}
and by $|B'\setminus  B|\le \e_3 |B|$, we have
\begin{equation}\label{eq:A'}
\overline{\deg}(v, B')\le \overline{\deg}\left(v,B \right)+|B'\setminus B| \binom{|B'|}{k-2} \le 2\e_2\binom{|B'|}{k-1} \text{ for any vertex }v\in A',
\end{equation}
and
\begin{equation}\label{eq:B'}
\deg(v, B')\le \deg\left(v,B \right)+|B'\setminus B| \binom{|B'|}{k-2} \le 2\e_1\binom{|B'|}{k-1} \text{ for any vertex }v\in B'.
\end{equation}
Moreover, for any $(k-1)$-set $S\subseteq B'$, by $\deg(S, A')+\deg(S, B')+\deg(S, V_0)\ge \delta_{k-1}(H)$ and $\overline{\deg}(S, A')=|A'| - \deg(S, A')$, we have
\[
\overline{\deg}(S, A') \le |A'| - \delta_{k-1}(H) + \deg(S, B') + \deg(S, V_0) \le \deg(S, B') + 3\e_3 |B|,
\]
where we used $\deg(S, V_0)\le |V_0|\le 2\e_3 |B|$, $|A'|\le \frac{a_1}mn + \e_3|B|$ and $\delta_{k-1}(H)\ge \frac{a_1}mn$.
Furthermore, for any $v\in B'$, by~\eqref{eq:B'}, we have
\[
\sum_{S: v\in S\in \binom{B'}{k-1}}{\deg}(S, B') = (k-1) \deg(v, B') \le 2(k-1)\e_1 \binom{|B'|}{k-1}.
\]
Putting these together gives that for any $v\in B'$,
\begin{align}
\overline{\deg}(v, A' B'^{k-1}) &= \sum_{S}\overline{\deg}(S, A') \le  \sum_{S}{\deg}(S, B') + 3\e_3|B|\binom{|B'|-1}{k-2} \le 2k\e_1 \binom{|B'|}{k-1}, \label{eq:ABB}
\end{align}
where the sums are on $S$ such that $v\in S\in \binom{B'}{k-1}$.
Let $\B$ be the set of $(k-1)$-sets $S\subseteq B'$ such that $\overline{\deg}_{H}(S, A') > \sqrt{2\e_2}|A'|$.
By~\eqref{eq:A'}, we have that $\overline{e}_H(A' B'^{k-1})\le 2\e_2 |A'|\binom{|B'|}{k-1}$ and thus
\begin{equation}\label{eq:familyB}
|\B|\le \sqrt{2\e_2}\binom{|B'|}{k-1}.
\end{equation}

In all three proofs we will define $\e_1, \e_2, \e_3$ and $A', B', V_0$ in the same way and thus Claim~\ref{clm:size} and~\eqref{eq:V_0} -- \eqref{eq:familyB} hold.

\subsection{Proof of Theorem \ref{lemE}}

Assume that $k\ge 3$, $a_1\le \cdots\le a_k$ and $m=a_1+\cdots +a_k$.
%Let $0<\xi\ll 1/m$, and let $n\in m \mathbb{N}$ be sufficiently large. 
%Let $H$ be a $k$-graph on $V$ of $n$ vertices such that $\delta_{k-1}(H) \ge \frac{a_1}{m}n$. Furthermore, assume that there is a set $B\subseteq V(H)$, such that $|B| = \frac{m-a_1}{m}n$ and $e(B)\le \xi \binom{|B|}{k}$. Let $A=V\setminus B$. Then $|A|= \frac{a_1}{m}n$.
%
Let $K:=K^{(k)}(a_1, \dots, a_k)$ such that $\gcd(K) = 1$.
Suppose $1/n \ll \xi \ll 1/m$ such that $n\in m\mathbb{N}$.
Assume $H$ is a $\xi$-extremal $k$-graph on $n$ vertices that satisfies \eqref{eq:deg}.
Let $B$ be a set of $\frac{m-a_1}{m}n$ vertices such that $e(B)\le \xi \binom{|B|}{k}$. Let $A=V\setminus B$. 
Define $\e_1, \e_2, \e_3, A', B', V_0$ as in Section 1.1 and thus Claim~\ref{clm:size} and~\eqref{eq:V_0} -- \eqref{eq:familyB} hold.

In the following proof we will build four vertex-disjoint $K$-tilings $\K_1, \K_2, \K_3, \K_4$, whose union is a $K$-factor of $H$. The ideal case is when $(m-a_1)|A'| = a_1 |B'|$ and $V_0 = \emptyset$ -- in this case we apply Lemma \ref{lem3} to obtain a $K$-factor of $H$ such that each copy of $K$ has $a_1$ vertices in $A'$ and $m-a_1$ vertices in $B'$. 
So the purpose of the $K$-tilings $\K_1, \K_2, \K_3$ is to cover the vertices of $V_0$ and adjust the sizes of $A'$ and $B'$ so that we can apply Lemma \ref{lem3} (and obtain $\K_4$) after $\K_1, \K_2, \K_3$ are removed.
More precisely, we cover the vertices of $V_0$ by $\K_2$ and let
\[
q := |B'| - |B| = \frac{a_1}mn - |A'|-|V_0|
\]
denote the \emph{discrepancy} between the current and ideal sizes of $B'$. If $q>0$, then we apply the minimum codegree condition to find copies of $K$ from $B'$. Since removing a copy of $K$ from $B'$ reduces the discrepancy by $a_1$, we can not reduce the discrepancy to zero unless $a_1$ divides $q$. Therefore we remove enough copies of $K$ from $B'$ (denoted by $\K_1$) such that the discrepancy is less than or equal to $-C$. This allows us to apply the definition of Frobenius numbers and remove more copies of $K$ (denoted by $\K_3$) to ``increase" the discrepancy to zero. 

\medskip
\noindent {\bf The $K$-tilings $\K_1, \K_2$.}  
Our goal is to find $K$-tilings $\K_1, \K_2$ such that $V_0\subseteq V(\K_2)$,
\begin{align}
&|\K_1| + |\K_2|\le 4\e_3 |B| \label{eq:K12} \text{ and }\\
& -2a_1 \e_3 |B| \le q_1 \le -C, \label{eq:t}
\end{align}
where $q_1:= \frac{a_1}{m}|V\setminus V(\K_1\cup \K_2)| - |A'\setminus V(\K_1\cup K_2)|$.

When $q\le -C$, let $\K_1=\emptyset$. When $q\ge 1-C$, $\K_1$ consists of $q+C$ copies of $K$ obtained from $H[B']$ as follows.\footnote{It suffices to find $\lceil (q+C)/a_1 \rceil$ copies of $K$ but since Fact~\ref{lemEX} provides (at least) $q+C$ copies, we choose to use all of them to simplify later calculations.}
Note that $\delta_{k-1}(H[B']) \ge \delta_{k-1}(H) - |A'| - |V_0| \ge q + C + f(n)$ by~\eqref{eq:deg} and the definition of~$q$. We claim that 
\begin{equation}
\label{eq:ex}
\frac{f(n)}{k} \binom{|B'|}{k-1} \ge {\rm ex} \Big(|B'| - (q+C-1)m, K\Big). 
\end{equation}
Indeed, when $q\ge 1$, we have $|B'| \ge |B|+1$ and $\binom{|B'|}{k-1} \ge \binom{|B|+1}{k-1}$. Since $f(n)\ge k\, \ex(|B|+1, K)/ \binom{B|+1}{k-1}$, it follows that $\frac{f(n)}{k} \binom{|B'|}{k-1} \ge \ex(|B|+1, K)$. Since $|B|+1= |B'|-q+1\ge |B'| - (q+C-1)m$, \eqref{eq:ex} follows from the monotonicity of the Tur\'an number. 
When $q\le 0$, we have $|B| - C< |B'|\le |B|$. The definition of $f(n)$ implies that $\frac{f(n)}{k} \binom{|B'|}{k-1} \ge \ex(|B'|, K)$. Again, \eqref{eq:ex} follows from the monotonicity of the Tur\'an number. 

By~\eqref{eq:ex}, we have
\begin{align*}
e_H(B')&\ge \frac{\delta_{k-1}(H[B'])}{k} \binom{|B'|}{k-1} \ge \frac{q+C}{k}\binom{|B'|}{k-1} + {\rm ex}(|B'| - (q+C-1)m, K) \\
&> (q+C-1)m\cdot 2\e_1 \binom{|B'|}{k-1} + {\rm ex}(|B'| - (q+C-1)m, K).
\end{align*}
By \eqref{eq:B'}, we can apply Fact~\ref{lemEX} to obtain $q+C$ vertex-disjoint copies of $K$ in $H[B']$, denoted by $\K_1$.

Next we choose a $K$-tiling $\K_2$ such that each copy of $K$ contains $a_1-1$ vertices in $A'$, one vertex in $V_0$ and $m-a_1$ vertices in $B'$.
By~\eqref{eq:V_0} and~\eqref{eq:familyB} we derive that for any vertex $w\in V_0$,
\[
|N(w, B')\setminus \B| \ge \frac{\e_1}2\binom{|B'|}{k-1} - \sqrt{2\e_2}\binom{|B'|}{k-1} \ge \frac{\e_1}3\binom{|B'|}{k-1}
\]
by the choice of $\e_1$ and $\e_2$. % YZ added
Let $V_0=\{w_1,\dots, w_{|V_0|}\}$.
For each $w_i$, by~\eqref{eq:erdos} we can find a copy $T_i$ of $K^{(k-1)}(a_2,\dots, a_k)$ in (the $(k-1)$-graph) $N(w, B')\setminus \B$ such that these copies are vertex disjoint, and are also vertex disjoint from $V(\K_1)$.
This is possible because the number of vertices in $B'$ that we need to avoid is at most $|V(\K_1)|+(m-a_1)|V_0|\le (\e_3|B|+C)\cdot m+ (m-a_1)\e_3|B|\le 2m\e_3|B|$, and so we have 
\[
|N(w, B')\setminus \B| - 2m\e_3 |B|\binom{|B'|}{k-2} \ge \frac{\e_1}3\binom{|B'|}{k-1} - 2m\e_3|B|\binom{|B'|}{k-2} \ge \frac{\e_1}4\binom{|B'|}{k-1}.
\]
Thus,~\eqref{eq:erdos} implies the existence of the desired $T_i$.
Note that each $\{w_i\}\cup T_i$ spans a copy of $K^{(k)}(1, a_2,\dots, a_k)$ in $H$.
To obtain copies of $K$, we extend each of them (one by one) by adding $a_1-1$ vertices from $A'$.
Note that each such vertex from $A'$ needs to be the common neighbor of $a':=\prod_{2\le i\le k}a_i$ $(k-1)$-sets, which is possible since by our choice, these $(k-1)$-sets are not in $\B$, and thus they have at least $(1 - a'\sqrt{2\e_2})|A'|$ common neighbors in $A'$.
Since $(1 - a'\sqrt{2\e_2})|A'| > (a_1-1)|V_0|$, we can greedily extend each $\{w_i\}\cup T_i$ into a copy of $K$.
Denote the resulting $K$-tiling by $\K_2$.

By definitions, we have $|\K_1|\le |q|+C$ and $|\K_2|=|V_0|$.
The result in~\cite{ErGr} (see also~\cite{Vitek75}) mentioned in Section 1 implies that $C\le (a_k-a_1)^2$.
By Claim~\ref{clm:size}, $|\K_1|+|\K_2|\le \e_3 |B| + C + 2\e_3 |B|\le 4\e_3 |B|$, i.e.,~\eqref{eq:K12} holds.
Let $A_1$ and $B_1$ be the sets of vertices in $A'$ and $B'$ not covered by $\K_1\cup \K_2$, respectively, and $V_1 := A_1\cup B_1$. 
So $q_1= \frac{a_1}{m}|V_1| - |A_1|$.
Note that $|A_1| = |A'| - (a_1-1)|V_0|$ and $|V_1| = n - m|\K_1| - m |V_0|$.
So we have
\[
q_1 = \frac{a_1}{m}n - a_1 |\K_1| - |V_0| - |A'| = q - a_1 |\K_1| \le q - |\K_1|.
\]
Recall that $|\K_1|=q+C$ if $q\ge 1-C$ and $|\K_1|=0$ if $q\le -C$.
So in both cases we get $q_1 \le q - |\K_1|\le - C$.
Moreover, by $-\e_3 |B|\le q\le \e_3 |B|$ and that $n$ is sufficiently large, we have $q_1 = q - a_1 |\K_1|\ge q - a_1 |q+C| \ge - 2a_1\e_3|B|$.
So~\eqref{eq:t} holds.

\medskip
\noindent {\bf The $K$-tiling $\K_3$.} Next we build our $K$-tiling $\K_3$. 
Since $-q_1\ge C > g(a_2 - a_1, a_3 - a_1,\dots, a_k - a_1)$ and $\gcd(a_2 - a_1, a_3 - a_1,\dots, a_k - a_1)=1$, there exists nonnegative integers $\ell_1,\dots, \ell_{k-1}$ such that
\[
\ell_1 (a_2 - a_1) + \ell_2 (a_3 - a_1) + \cdots + \ell_{k-1} (a_k - a_{1}) = -q_1
\]
(here we let $\ell_i=0$ if $a_i - a_1=0$).
For each $i\in [k-1]$, we pick $\ell_i$ vertex disjoint copies of $K$ each with $a_{i+1}$ vertices in $A_1$ and $m-a_{i+1}$ vertices in $B_1$ such that all copies of $K$ are vertex disjoint.
Denote by $\K_3$ as the set of copies of $K$.
Note that we can choose these desired copies of $K$ by Proposition~\ref{supersaturation} and the fact that $H(A_1 B_1^{k-1})$ is dense (see~\eqref{eq:abb}).
By the definition of $\K_3$, we have $|\K_3| = \sum_{i\in [k-1]} \ell_i$ and
\begin{align*}
\frac{a_1}m|V_1\setminus V(\K_3)| - |A_1\setminus V(\K_3)| = &\, \frac{a_1}m (|V_1|- m|\K_3|) - (|A_1| - |V(\K_3)\cap A_1|)\\
= &\, q_1 + (|V(\K_3)\cap A_1| - a_1 |\K_3|) \\
=&\, q_1 + \sum_{i\in [k-1]} \ell_i (a_{i+1} - a_1) = 0.
\end{align*}
Let $A_2=A_1\setminus V(\K_3)$ and $B_2=B_1\setminus V(\K_3)$.
So we have $\frac{a_1}m (|A_2|+ |B_2|)=|A_2|$, i.e., $(m-a_1)|A_2|=a_1 |B_2|$.

\medskip
Note that $|\K_3| = \sum_{i\in [k-1]} \ell_i \le -q_1 \le 2a_1\e_3 |B|$ by~\eqref{eq:t}. 
Together with~\eqref{eq:K12} and Claim~\ref{clm:size}, we have 
\[
|B_2|\ge |B'| - |V(\K_1\cup \K_2)| - m|\K_3|\ge |B'| - 4m\e_3 |B| - 2a_1 m \e_3|B|> (1-\e_1)|B'|.
\]
Hence, for every vertex $v\in A_2$, by~\eqref{eq:A'},
\[
\overline{\deg}(v, B_2)\le \overline{\deg}(v, B')\le 2\e_2\binom{|B'|}{k-1}\le 2\e_2 \binom {\frac{1}{1-\e_1}|B_2|}{k-1} < \e_1 \binom{|B_2|}{k-1}.
\]
By~\eqref{eq:ABB} and $|B_2| \ge (1-\e_1)|B'|$, for every $v\in B_2$ we have
\begin{align}
\overline{\deg}(v, A_2 B_2^{k-1})\le \overline{\deg}(v, A' B'^{k-1}) \le 2k\e_1 \binom{|B'|}{k-1}\le 3k\e_1\binom{|B_2|}{k-1}. \label{eq:abb}
\end{align}

\medskip
\noindent {\bf The $K$-tiling $\K_4$.} 
At last, we apply Lemma \ref{lem3} with $X=A_2$, $Y=B_2$ and $\rho=3k{\e_1}$ and get a $K$-factor $\K_4$ on $A_2\cup B_2$.

So $\K_1\cup \K_2\cup \K_3\cup \K_4$ is a $K$-factor of $H$.
This concludes the proof of Theorem \ref{lemE}.

\subsection{Proof of Theorem~\ref{lemE3}}
As mentioned in Section~1, \eqref{eq:deg} reduces to $\delta_{k-1}(H)\ge \frac{n}{k+1}+1$ when $K= K^{(k)}(1,\dots, 1,2)$.
Thus Theorem~\ref{lemE3} follows from Theorem~\ref{lemE} if Condition (i) holds. Now assume Condition (ii), that is, 
$\delta_{k-1}(H)\ge \frac{n}{k+1}$ and $k-i \nmid\binom{n'-i}{k-1-i}$ for some $0\le i\le k-2$ and $n' = \frac{k n}{k+1}+1$.
The proof follows the proof of Theorem~\ref{lemE} (with $C=0$) closely and the only difference is the existence of $\K_1$ when $q=|B'| - |B|\ge 1$.
Note that $\delta_{k-1}(H[B']) \ge q$.
Since ${\rm ex}(n, K^{(k)}(1,\dots, 1,2))\le \binom{n}{k-1}/k$, when $q\ge 2$ we can find $q$ copies of $K^{(k)}(1,\dots, 1,2)$ in $B'$ by Fact~\ref{lemEX}.
Otherwise $q=1$, i.e., $|B'|=n'$. Assume to the contrary that $H[B']$ is $K^{(k)}(1,\dots, 1,2)$-free, i.e., every $k-1$ vertices in $B'$ has degree at most $1$ in $B'$.
Then by $\delta_{k-1}(H[B']) \ge 1$ we derive that every $k-1$ vertices in $B'$ has degree exactly $1$ in $B'$.
This means that a Steiner system $S(k-1, k, n')$ exists, contradicting the divisibility conditions.

\subsection{Proof of Theorem~\ref{lemE2}}
\label{sec:Elc}

Recall that a loose cycle $C_s^k$ has a vertex set $[s(k-1)]$ and $s$ edges $\{\{j(k-1)+1,\dots, j(k-1)+k\} \text{ for } 0\le j <s\}$, where we treat $s(k-1)+1$ as $1$. When $s=2, 3$, $C_s^k$ has a unique $k$-partite realization: $C_2^k\subset K^{(k)}(1,1, 2, \dots, 2)$ and $C_3^k \subset K^{(k)}(2, 2, 2, 3, \dots, 3)$. When $s\ge 4$, we 3-color the vertices $j(k-1)+1$, $0\le j < s$ (these are the vertices of degree two) with $\lfloor s/2 \rfloor$ red vertices, $\lfloor s/2 \rfloor-1$ blue vertices and the remaining one or two vertices green. We complete the $k$-coloring of $C_s^k$ by coloring the $k-2$ uncolored vertices of each edge of $C_s^k$ with the $k-2$ colors not used to color the two vertices of degree two. In this coloring, there are  $\lceil s/2 \rceil$ red vertices, $\lceil s/2 \rceil+1$ blue vertices, $s-1$ or $s-2$ green vertices, and $s$ vertices in other color classes. Furthermore, since each vertex in $C_s^k$ has degree at most $2$, in any $k$-coloring of $C_s^k$, each color class has size at least $\lceil s/2 \rceil$. Thus, $\sigma(C_s^k) = \frac{\lceil s/2 \rceil}{s(k-1)}$. 

We summarize above arguments into a proposition.

\begin{proposition}\label{propcsk}
For any $k\ge 4$ and $s\ge 2$ we have $\sigma(C_s^k) = \frac{\lceil s/2 \rceil}{s(k-1)}$. % YZ removed \tau(C_s^k) because it is not the same as \sigma.
Moreover, there exists a $k$-partite realization of $C_s^k$, in which the smallest part is of size $\lceil s/2 \rceil$ and there is a part of size $\lceil s/2 \rceil +1$.
In particular, $\gcd(C_s^k)=1$. \qed
\end{proposition}

%The value of ${\rm ex}(n, C_s^k)$ is determined in \cite{FurediJiang, KMV}, which is $(\lfloor \frac{s-1}{2} \rfloor+o(1)) \binom{n}{k-1}$.
%So by Theorem~\ref{lemE} we get that if $H$ is extremal, then $\delta_{k-1}(H)\ge \sigma(C_s^k)n+ \lfloor \frac{s-1}{2} \rfloor k+1$ guarantees a $C_s^k$-factor.
%We shall improve this to $\delta_{k-1}(H)\ge \sigma(C_s^k)n$.
%In fact, our later proof shows that in most cases we do not need to know ${\rm ex}(n, C_s^k)$: 

In order to prove Theorem~\ref{lemE2}, we use upper bounds for ${\rm ex}(n, P_2^k)$ and ${\rm ex}(n, C_2^k)$ from~\cite{FF1}.
% \footnote{In fact, the exact value of ${\rm ex}(n, C_2^k)$ is known but Theorem~\ref{thm:FF1} suffices for our purpose.}.
Note that the results in~\cite{FF1} are in the language of extremal set theory, but it is easy to formalize the results for our purpose: a $k$-graph is $C_2^k$-free if and only if the size of the intersection of any two edges is not $2$; a $k$-graph is $P_2^k$-free if and only if the size of the intersection of any two edges is not $1$.

\begin{theorem}\cite{FF1}\label{thm:FF1}
For $k\ge 4$, there exists a constant $d_k$ such that ${\rm ex}(n, C_2^k) \le d_k n^{\max\{2,k-3\}}$ and ${\rm ex}(n, P_2^k) \le \binom{n-2}{k-2}$.
\end{theorem}

\begin{proof}[Proof of Theorem~\ref{lemE2}]
Let $K^{(k)}(a_1, \dots, a_k)$ be a $k$-partite realization of $C_s^k$ such that $a_1=\lceil s/2 \rceil$ and $a_{k'}=\lceil s/2 \rceil+1$ for some $k'\in [k]$.
Then we have that $\gcd(K^{(k)}(a_1, \dots, a_k))=1$ and $C=g(a_2-a_1, \dots, a_k-a_1)+1=0$.
Let $m=s(k-1)$.
Suppose $1/n \ll \xi \ll 1/m$.
Assume $H$ is an $n$-vertex $k$-graph which is $\xi$-extremal and $\delta_{k-1}(H)\ge \frac{a_1}{m}n$.
Define $\e_1, \e_2, \e_3, A', B', V_0$ as in Section 1.1 and thus Claim~\ref{clm:size} and~\eqref{eq:V_0} -- \eqref{eq:familyB} hold.

The proof follows the one of Theorem~\ref{lemE} by constructing $C_s^k$-tilings $\K_1$, $\K_2$, $\K_3$ and $\K_4$, whose union forms a perfect $C_s^k$-tiling of $H$.
We will only show the first step, the existence of $\K_1, \K_2$, because it is the only part different from that in the proof of Theorem~\ref{lemE}.

We here need a stronger control on the `good' $(k-1)$-sets in $B'$, i.e., every vertex in $B'$ is in many such good $(k-1)$-sets (note that this is stronger than $\mathcal{B}$, which we only have control on the total number of `bad' sets).
Let $G$ be the $(k-1)$-graph on $B'$ whose edges are all $(k-1)$-sets $S\subseteq B'$ such that $\overline{\deg}_{H}(S, A') < \sqrt{2k\e_1}|A'|$. 
We claim that
\begin{equation}\label{eq:dG0}
\delta_1(G) \ge (1- m\sqrt{2k\e_1}) \binom{|B'|-1}{k-2}, \text{ and thus, } \overline{e}(G)\le m\sqrt{2k\e_1} \binom{|B'|}{k-1}.
\end{equation}
Suppose instead, some vertex $v\in B'$ satisfies $\overline \deg_G(v)> m\sqrt{2k\e_1} \binom{|B'|-1}{k-2}$. 
Since every non-neighbor $S'$ of $v$ in $G$ satisfies $\overline \deg_H(S' \cup\{v\}, A')\ge \sqrt{2k\e_1} |A'|$, we have
\[
\overline \deg_{H}(v, A'B'^{k-1}) > m\sqrt{2k\e_1} \binom{|B'|-1}{k-2} \sqrt{2k\e_1} |A'| > 2k\e_1 \binom{|B'|}{k-1},
\]
where we used $m|A'|>|B'|$. This contradicts~\eqref{eq:ABB}.

\medskip
\noindent {\bf The $K$-tilings $\K_1, \K_2$.} 
Assume that $q=|B'| - \frac{m-a_1}{m}n$.
Similar as in the proof of Theorem~\ref{lemE}, our goal is to find $C_s^k$-tilings $\K_1, \K_2$ such that $V_0\subseteq V(\K_2)$, and~\eqref{eq:K12} and~\eqref{eq:t} hold (with $C=0$).

We first construct a $C_s^k$-tiling $\K_1$ such that $|\K_1|=\max\{q, 0\}$ and each copy of $C_s^k$ in $\K_1$ contains exactly $m-a_1+1$ vertices in $B'$.
Let $\K_1=\emptyset$ if $q\le 0$.
Then assume $q\ge 1$ and note that $\delta_{k-1}(H[B'])\ge q$.
Thus
\[
e_H(B')\ge \frac1k \delta_{k-1}(H[B'])\binom{|B'|}{k-1} \ge \frac{q}{k}\binom{|B'|}{k-1} > (q-1) m\cdot 2\e_1 \binom{|B'|}{k-1} + \frac{q}{2k}\binom{|B'|}{k-1}.
\]
Since $k\ge 4$, by Theorem~\ref{thm:FF1}, we know that ${\rm ex}(|B'|, C_2^k)\le \frac{q}{2k}\binom{|B'|}{k-1}$ and ${\rm ex}(|B'|, P_2^k)\le \frac{q}{2k}\binom{|B'|}{k-1}$.
First assume $s=2$. 
Note that if $s=2$, then $a_1=1$, that is, $m-a_1+1=m$.
By \eqref{eq:B'} and Fact~\ref{lemEX}, $H[B']$ contains a set of $q$ vertex disjoint copies of $C_2^k$. 
Denote it by $\K_1$ and we are done.
Second assume $s\ge 3$, then by \eqref{eq:B'} and Fact~\ref{lemEX}, $H[B']$ contains a collection of $q$ vertex disjoint copies of $P_2^k$ denoted by $Q_{1},\dots, Q_{q}$.

For each $i\in [q]$, we extend $Q_i$ (or only one edge of $Q_i$) to a copy of $C_s^k$ such that all copies of $C_s^k$ are vertex disjoint and each copy contains exactly $m-a_1+1$ vertices in $B'$.
Indeed, for $i\in [q]$, let $E_1, E_s$ be the two edges in $Q_i$.
If $s$ is even, let $Q_i'=Q_i$ and pick $u\in E_1\setminus E_s$, $u'\in E_s\setminus E_1$; if $s$ is odd, let $Q_i'=E_1$ and let $u, u'$ be two distinct vertices from $Q_i'$.
Let $s'=2 \lceil s/2 \rceil $.
We pick vertex sets $S_2,\dots, S_{s'-1}$ of size $k-2$, and vertices $u_3,u_5,\dots, u_{s'-3}$ from the unused vertices in $B'$ such that the following $(k-1)$-sets
\begin{align*}
&F_2:=S_2\cup \{u\}, \, F_{s'-1}:= S_{s'-1}\cup \{u'\}, \\
&F_{2j-1}:=S_{2j-1}\cup \{u_{2j-1}\}, \, F_{2j}:=S_{2j}\cup \{u_{2j-1}\} \text{ for } 2\le j\le (s'-2)/2
\end{align*}
are in $E(G)$. This is possible by~\eqref{eq:dG0} (we pick an edge that contains $u$, an edge that contains $u'$ and then some copies of $P_2^{k-1}$ such that all these are vertex disjoint and vertex disjoint from other existing vertices).
Then for each $2\le j\le s'/2$, pick $u_{2j-2}\in N_H(F_{2j-2})\cap N_H(F_{2j-1})\cap A'$, which is possible since $F_i\in E(G)$ and thus $\overline{\deg}_H(F_i, A')< \sqrt{2k\e_1}|A'|$.
Note that $Q_i'\cup \bigcup_{2\le j\le s'-1} S_{j}\cup \{u_2,\dots, u_{s'-2}\}$ spans a loose cycle of length $s'-1$ if $s$ is odd and $s'$ if $s$ is even, i.e., it spans a copy of $C_s^k$.
Moreover, each such copy contains exactly $s'/2-1=a_1-1$ vertices in $A'$, and thus exactly $m-a_1+1$ vertices in $B'$.

Next we choose a $K$-tiling $\K_2$ such that each copy of $K$ contains $a_1-1$ vertices in $A'$, one vertex in $V_0$ and $m-a_1$ vertices in $B'$.
This can be done by the same argument as in the proof of Theorem~\ref{lemE}.

By definitions, we have $|\K_1|\le |q|$ and $|\K_2|=|V_0|$.
By Claim~\ref{clm:size}, $|\K_1|+|\K_2|\le |q| + |V_0|\le 4\e_3 |B|$, i.e.,~\eqref{eq:K12} holds.
Let $A_1$ and $B_1$ be the sets of vertices in $A'$ and $B'$ not covered by $\K_1\cup \K_2$, respectively. 
Let $V_1 := A_1\cup B_1$.
Let $q_1= \frac{a_1}m |V_1| - |A_1|$.
Recall that if $s=2$, then $a_1=1$.
Therefore for any $s\ge 2$, by the definitions of $A_1, B_1$, we have $|A_1| = |A'| - (a_1-1)|\K_1|- (a_1-1)|V_0|$ and $|V_1| = n - m|\K_1| - m |V_0|$.
 Together with $q= \frac{a_1}mn - (|A'|+|V_0|)$, we get
\[
q_1 = \frac{a_1}{m}n - |\K_1| - |V_0| - |A'| = q - |\K_1| = q - \max\{0, q\}\le 0.
\]
Since $-\e_3 |B|\le q\le \e_3 |B|$, we get $q_1= q - \max\{0, q\} \ge -|q| \ge -2 a_1\e_3|B|$.
Thus~\eqref{eq:t} holds.
The rest of the proof is similar to the previous and is omitted.
\end{proof}

\subsection{Proof of Lemma~\ref{lem3}}
In this subsection we prove Lemma~\ref{lem3} by following the proof of \cite[Lemma 4.4]{HZ3}, which proves the case when $K= K^{(3)}(1, 1, 2)$.

We need the following result of Lu and Sz\'ekely~\cite[Theorem 3]{LuSz}.
\begin{theorem}\cite{LuSz}\label{thm:LuSz}
Let $F$ be a $k$-graph in which each edge intersects at most $d$ other edges. If $H$ is an $n$-vertex $k$-graph such that $|V(F)|$ divides $n$ and
\[
\delta_1(H)\ge \left( 1 -  \frac{1}{e(d+1+x k^2)} \right) \binom{n-1}{k-1},
\]
where $e=2.718...$ and $x=|E(F)|/|V(F)|$, then $H$ contains an $F$-factor.
\end{theorem}

\begin{proof}[Proof of Lemma \ref{lem3}]
Let $t=|X|/a_1$.
Let $\G$ be the $(k-1)$-graph on $Y$ whose edges are all $(k-1)$-sets $S\subseteq Y$ such that $\overline{\deg}_{H}(S, X) < \sqrt{\rho}t$. 
First we claim that
\begin{equation}\label{eq:dG}
\delta_1(\G) \ge (1- m\sqrt{\rho}) \binom{|Y|-1}{k-2},
\end{equation}
and consequently,
\begin{equation}\label{eq:eG}
\overline{e}(\G)\le m\sqrt{\rho} \binom{|Y|}{k-1}.
\end{equation}
Suppose instead, some vertex $v\in Y$ satisfies $\overline \deg_\G(v)> m\sqrt{\rho} \binom{|Y|-1}{k-2}$. 
Since every non-neighbor $S'$ of $v$ in $\G$ satisfies $\overline \deg_H(S' \cup\{v\}, X)\ge \sqrt{\rho} t$, we have
$\overline \deg_{H}(v, XY^{k-1}) > m\sqrt{\rho} \binom{|Y|-1}{k-2} \sqrt{\rho} t$.
Since $|Y| = (m-a_1)t$, we have
\[
\overline \deg_{H}(v, XY^{k-1}) > m\rho\frac{|Y|}{m-a_1} \binom{|Y|-1}{k-2}
> \rho \binom{|Y|}{k-1},
\]
contradicting our assumption.

Let $Q$ be an $(m-a_1)$-subset of $Y$.  We call $Q$  \emph{good} (otherwise \emph{bad}) if every $(k-1)$-subset of $Q$ is an edge of $\G$, i.e., $Q$ spans a \emph{clique} of size $m-a_1$ in $\G$.
Furthermore, we say $Q$ is \emph{suitable} for a vertex $x\in X$ if $x\cup T\in E(H)$ for every $(k-1)$-set $T\subset Q$. 
Note that if an $(m-a_1)$-set is good, by the definition of $\G$, it is suitable for at least $(1- \binom{m-a_1}{k-1}\sqrt{\rho})t$ vertices of $X$.

\begin{claim}\label{clm:suitablesets}
For any $x\in X$, at least $(1-\rho^{1/4} )\binom{|Y|}{m-a_1}$ $(m-a_1)$-subsets of $Y$ are good and suitable for $x$.
\end{claim}

\begin{proof}
First by the degree condition of $H$, namely, for any $x\in X$, the number of $(m-a_1)$-sets in $Y$ that are not suitable for $x$ is at most ${\rho}\binom{|Y|}{k-1}\binom{|Y|-k+1}{m-a_1-k+1}\le \sqrt{\rho}\binom{|Y|}{m-a_1}$.
Second, by \eqref{eq:eG}, at most
\[
\overline{e}(\G) \binom{ |Y|-k+1 }{m-a_1-k+1} \le m\sqrt{\rho} \binom{|Y|}{k-1} \binom{ |Y|-k+1 }{m-a_1-k+1} \le \frac12\rho^{1/4}\binom{|Y|}{m-a_1}
\]
$(m-a_1)$-subsets of $Y$ contain a non-edge of $\G$.
Since $\rho^{1/2} + \frac12 \rho^{1/4}\le \rho^{1/4}$, the claim follows.
\end{proof}

Let $\F_0$ be the set of good $(m-a_1)$-sets in $Y$.
We will pick a family of disjoint good $(m-a_1)$-sets in $Y$ such that for any $x\in X$, many members of this family are suitable for $x$. To achieve this, we pick a family $\mathcal F$ by selecting each member of $\F_0$ randomly and independently with probability $p=4\binom{m-a_1}{k-1}\sqrt{\rho} |Y|/\binom{|Y|}{m-a_1}$.
Then $|\F|$ follows the binomial distribution $B(|\F_0|, p)$ with expectation $\mathbb{E}(|\F|) = p |\F_0| \le p\binom{|Y|}{m-a_1}$.
Furthermore, for every $x\in X$, let $f(x)$ denote the number of members of $\F$ that are suitable for $x$. Then $f(x)$ follows the binomial distribution $B(N, p)$ with $N\ge (1-\rho^{1/4} )\binom{|Y|}{m-a_1}$ by Claim \ref{clm:suitablesets}. Hence $\mathbb{E}(f(x))\ge p(1-\rho^{1/4} )\binom{|Y|}{m-a_1}$. 
Since there are at most $\binom{|Y|}{m-a_1}\cdot (m-a_1)\cdot \binom{|Y|-1}{m-a_1-1}$ pairs of intersecting $(m-a_1)$-sets in $Y$,
the expected number of intersecting pairs of $(m-a_1)$-sets in $\mathcal F$ is at most
\[
p^2 \binom{|Y|}{m-a_1}\cdot (m-a_1)\cdot \binom{|Y|-1}{m-a_1-1}=16\binom{m-a_1}{k-1}^2(m-a_1)^2 {\rho} |Y|.
\]

By Chernoff's bound (the first two properties) and Markov's bound (the last one), we can find a family $\mathcal F$ of good $(m-a_1)$-subsets of $Y$ that satisfies
\begin{itemize}
\item
$|\mathcal F|\le 2p \binom{|Y|}{m-a_1} \le 8\binom{m-a_1}{k-1}\sqrt{\rho} |Y|$,
\item
for any vertex $x\in X$, at least
$\frac{p}2 (1- \rho^{1/4})\binom{|Y|}{m-a_1}  = 2\binom{m-a_1}{k-1}(1-\rho^{1/4})\sqrt{\rho}|Y|$
members of $\mathcal F$ are suitable for $x$.
\item
the number of intersecting pairs of $(m-a_1)$-sets in $\mathcal F$ is at most $32\binom{m-a_1}{k-1}^2(m-a_1)^2 {\rho} |Y|$.
\end{itemize}
After deleting one $(m-a_1)$-set from each of the intersecting pairs from $\F$, we obtain a family $\mathcal F'\subseteq \F$ consisting of at most $8\binom{m-a_1}{k-1}\sqrt{\rho} |Y|$ disjoint good $(m-a_1)$-subsets of $Y$ and for each $x\in X$, at least
\begin{equation}\label{eq:F'}
2\binom{m-a_1}{k-1}(1-\rho^{1/4})\sqrt{\rho}|Y| - 32\binom{m-a_1}{k-1}^2(m-a_1)^2 {\rho} |Y|\ge \binom{m-a_1}{k-1}\sqrt{\rho}|Y|
\end{equation}
members of $\mathcal F'$ are suitable for $x$.

Denote $\F'$ by $\{Q_1, Q_2, \dots, Q_{q} \}$ for  some $q\le 8\binom{m-a_1}{k-1}\sqrt{\rho} |Y|$. 
Let $Y_1= Y \setminus V(\F') $ and $\G'=\G[Y_1]$. Then $|Y_1|=|Y|-(m-a_1)q$. Since $\overline \deg_{\G'}(v)\le \overline \deg_\G(v)$ for every $v\in Y_1$, we have, by \eqref{eq:dG},
\[
\delta_1(\G') \ge \binom{|Y_1|-1}{k-2} -  m\sqrt{\rho} \binom{|Y|-1}{k-2}\ge (1-2m\sqrt{\rho})\binom{|Y_1|-1}{k-2}.
\]
By the choice of $\rho$ and Theorem~\ref{thm:LuSz}, $\G'$ contains a perfect tiling $\{Q_{q+1},\dots, Q_t\}$ such that each $Q_i$ is a clique on $m-a_1$ vertices for $q+1\le i\le t$.

Consider the bipartite graph $\Gamma$ between $X$ and $\Q:= \{Q_1, Q_2, \dots, Q_{t}\}$ such that $x\in X$ and $Q_i\in \Q$ are adjacent if and only if $\Q_i$ is suitable for $x$. 
For every $i\in [t]$, since each $Q_i$ is a clique in $\G$, we have $\deg_{\Gamma} (Q_i) \ge |X| - \binom{m-a_1}{k-1}\sqrt \rho t$ by the definition of $\G$. 
Let $\Q'= \{Q_{q+1}, \dots, Q_{t}\}$ and $X_0$ be the set of $x\in X$ such that $\deg_{\Gamma}(x, \Q')\le |\Q'|/2$. Then
\[
|X_0| \frac{|\Q'|}2 \le \sum_{x\in X} \overline{\deg}_{\Gamma}(x, \Q') \le \binom{m-a_1}{k-1} \sqrt \rho t\cdot |\Q'|,
\]
which implies that
$|X_0| \le 2\binom{m-a_1}{k-1} \sqrt\rho t = 2\binom{m-a_1}{k-1}\sqrt{\rho} \frac{|Y|}{m-a_1}\le \binom{m-a_1}{k-1} \sqrt\rho |Y|$ (since $m-a_1\ge 2$).

We now find a perfect tiling of $K^{(2)}(1,a_1)$ in $\Gamma$ such that the center of each $K^{(2)}(1,a_1)$ is in $\Q$.

\begin{enumerate}
\item[Step 1:]
Each $x\in X_0$ is matched to some $Q_{i}$, $i\in [q]$ that is suitable for $x$ -- this is possible because of \eqref{eq:F'} and
$|X_0| \le \binom{m-a_1}{k-1} \sqrt\rho |Y|$.
\item[Step 2:]
Each $Q_{i}$, $i\in [q]$ is matched with $a_1-1$ or $a_1$ more vertices in $X\setminus X_0$ -- this is possible because $\deg_{\Gamma} (Q_i) \ge |X| - \binom{m-a_1}{k-1}\sqrt \rho t \ge |X_0| + a_1 q$.
Thus, all $\Q_i$, $i\in [q]$ are covered by vertex-disjoint copies of $K^{(2)}(1,a_1)$.
\item[Step 3:]
Let $X'$ be the set of uncovered vertices in $X$ and note that $|X'|=a_1 t - a_1 q = a_1 |\Q'|$.
Partition $X'$ arbitrarily into $X_1, \dots, X_{a_1}$ each of size $|\Q'|$.
Note that for each $i\in [a_1]$, we have for all $x\in X_i$, $\deg_{\Gamma}(x, \Q')> |\Q'|/2$ and for all $Q_j\in \Q'$, $\deg_{\Gamma}(Q_j, X_i) \ge |X_i| - \binom{m-a_1}{k-1}\sqrt \rho t \ge |X_i|/2$.
So the Marriage Theorem provides a perfect matching in each $\Gamma[X_i, \Q']$, $i\in [a_1]$ and thus we get a perfect tiling of $K^{(2)}(1,a_1)$ on $X'\cup \Q'$ in $\Gamma$.
\end{enumerate}
The perfect tiling of $K^{(2)}(1,a_1)$ in $\Gamma$ gives rise to the desired $K$-factor in $H$.
\end{proof}

\section{Concluding Remarks}

In this paper we study the minimum codegree threshold $\delta(n, K)$ for tiling complete $k$-partite $k$-graphs $K$ perfectly when $\gcd(K)=1$.
By Proposition~\ref{disprove}, $\delta(n, K)\ge n/m + {\rm coex}(\frac{m-1}{m}n+1, K)$ when $a_1=1$.
In view of this and Theorem~\ref{thm:main}, it is interesting to know if one can replace the second term in~\eqref{eq:deg} by a term similar to ${\rm coex}(\frac{m-a_1}{m}n+1, K)$.
Moreover, it is interesting to know if ${\rm ex}(n, K)/\binom{n}{k-1}$ (or ${\rm coex}(n, K)$) is monotone on $n$ so that the maximization in~\eqref{eq:deg} could be avoided.

Following the notion in \cite{My14}, we call $k$-partite $k$-graphs satisfying the three lines of~\eqref{eq:nk1} type 0, type 1 and type $d$, respectively.
For complete $k$-partite $k$-graphs $K$ of type $d$ for an even $d\ge 0$, a simple application of the absorbing method together with Lemma~\ref{lemK} implies that $\delta(n, K) = n/2 + o(n)$, which gives a reproof of the result of Mycroft~\cite{My14} without using the Hypergraph Blow-up Lemma.
We think that a further sharpening is possible by a careful analysis on the extremal case, to which we shall return in the near future. 

Suppose $K := K^{(k)}(a_1,\dots, a_k)$. When $K$ is type 1 (namely, $\gcd(K)=1$), Proposition~\ref{disprove} and Theorem~\ref{thm:112} settle Conjecture~\ref{conj:Mycroft} (either negatively or positively).
Now we give a construction showing that Conjecture~\ref{conj:Mycroft} is false for many other complete $k$-partite $k$-graphs, for which we need to recall a construction by Mycroft~\cite{My14}.
%For any integer $k\ge 3$, let $p$ be the smallest prime factor of $k$.
Fix a prime number $p$. %and consider the following vectors in $\mathbb{Z}_p^p$.
Let $\bfu_i\in \mathbb{Z}_p^p$ be the unit vector whose $i$th coordinate is one.  
For $1\le i<p$, let $\bfv_i=\bfu_i + (i-1)\bfu_p$.
Let $L$ be the (proper) sublattice of $\mathbb{Z}_p^p$ generated by $\bfv_1,\dots, \bfv_{p-1}$.
The following property was proved in~\cite[Section 2]{My14}.
% and we include a proof for completeness.
\begin{itemize}\label{claim:L}
\item[($\dagger$)] For any vector $\bfv\in \mathbb{Z}_p^p$, there exists precisely one $i\in [p]$ such that $\bfv + \bfu_i\in L$.
\end{itemize}
Let $\cP=\{V_1, V_2, \dots, V_p\}$ be a partition of $V$ such that $|V_1|+|V_2|+\cdots +|V_p|=n$, $|V_i| = n/p \pm 1$ for $i\in [p]$ and $\bfi_\cP(V)\notin L \pmod{p} $ (recall that $\bfi_\cP(S)$ is the vector of $\mathbb{Z}^p$ whose $i$th coordinate is $|S\cap V_i|$).
%(on a co-ordinate by co-ordinate basis).
%Note that the restriction on the sizes is possible because $L$ is transferral-free: either $(n/p, \dots, n/p)$ or $(n/p-1, n/p+1, n/p,\dots, n/p)$ is not in $L$ modulo $p$.
Let $H_p$ be the $k$-graph on $V$ whose edges are $k$-sets $e$ such that $\bfi_{\cP}(e)\in L \pmod{p}$.
Observe that ($\dagger$) implies that $\delta_{k-1}(H_p) \ge n/p-k$.
%Indeed, for any $(k-1)$-set of vertices $S$, by the second part of Claim~\ref{claim:L}, there exists $i\in [p]$ such that $\bfi_{\cP}(S) + \bfu_i\in L$.
%By the definition of $H_k$, we have $V_i\setminus S\subseteq N_{H_k}(S)$ and thus $\deg_{H_k}(S)\ge |V_i\setminus S| \ge n/p-1-(k-1)=n/p-k$.

\begin{proposition}\cite{My14}
\label{pro:Myc}
Suppose that $K$ is a complete $k$-partite $k$-graph of type $d$ for some $d\ne 1$. 
Let $p$ be the smallest prime factor of $d$ (thus $p=2$ when $d=0$).
Then $H_p$ contains no $K$-factor. 
\end{proposition}

Using $H_p$ and the construction behind Proposition~\ref{prop:lb}, we disprove Conjecture~\ref{conj:Mycroft} for many complete $k$-partite $k$-graphs of type $d\ne 1$. 
\begin{proposition}\label{disprove2}
Let $K := K^{(k)}(a_1,\dots, a_k)$ be type $d\ne 1$ and let $p$ be the smallest prime factor of $d$.
Then Conjecture~\ref{conj:Mycroft} is false for $K$ if $a_{k-2}\ge p+1$. 
%and for $K$ of type $d\ge 2$ if $a_{k-2}\ge p+1$, where $p$ is 
\end{proposition}

\begin{proof}
%Let $p$ be the smallest prime factor of $d$;
%We define $p=p(d)$ for an integer $d\ge 2$ or $d=0$ as follows.
%If $d\ge 2$, let $p$ be the smallest prime factor of $d$; if $d=0$, let $p=2$.
%Now assume $K := K^{(k)}(a_1,\dots, a_k)$ is such that $K$ is type 0 or type $d\ge 2$ and $a_{k-2}\ge p+1$.
Let $n\in (a_1+\cdots+a_k) \mathbb{Z}$ and let $G$ be an $n$-vertex $K^{(k)}(1,\dots,1,2,2)$-free $k$-graph with $\delta_{k-1}(G) = (1-o(1))\sqrt{n}$ provided by Proposition~\ref{prop:lb}.
We first take a copy of an $n$-vertex $k$-graph $H_p$ under a vertex partition $V_1, V_2, \dots, V_p$.
Then we take a random permutation of $V(G)$ and then add $E(G)$ on top of $H_p$. 
Denote the resulting graph by $H$.
%Let $A$ be a random subset of $V(G)$ such that $|A|=n/2$ or $n/2-1$ and $a\nmid |A|$, and let $B = V(G)\setminus A$.
%Let $H$ be a $k$-graph on $V(G)$ with $E(H) = E(G)\cup E_{even}$, where $E_{even}$ (or $E_{odd}$) consists of all $k$-sets $e\subset V(G)$ such that $|e\cap A|$ is even (or odd).
By standard concentration results, we have $\delta_{k-1}(H) \ge n/p + (1-o(1))\sqrt{n}/p$.
%It is not hard to see that if a copy of $K$ in $H$ contains an edge in $E(G)\setminus E_{even}$, then it contains a copy of $K_{1,\dots,1,2,2}$ in $E(G)\setminus E_{even}$, which is impossible because $G$ is $K_{1,\dots,1,2,2}$-free. 
We claim that 
\begin{itemize}
\item[($\ddagger$)] {every copy of $K$ in $H$ has each of its vertex classes completely in $V_i$ for some $i\in [p]$.}
\end{itemize}
%Note that if this holds, since $G$ is $K$-free and $K$ is a blow-up of an edge, we infer that all copies of $K$ in $H$ are actually in $H_p$.
Suppose instead, there exist distinct $i_1, i_2\in [p]$ and a copy of $K$ in $H$ with vertex classes $U_1,\dots, U_k$ such that $U_l\cap V_{i_1} \ne \emptyset$ and $U_l\cap V_{i_2} \ne \emptyset$ for some $l\in [k]$. 
%Without loss of generality, assume $l=1$.
%For $i,$et $A_l:=U_l\cap A$, $B_l:=U_l\cap B$. 
For $i\in [k]\setminus \{l\}$, let $C_i= U_i\cap V_j$ such that $|U_i\cap V_j|\ge |U_i\cap V_{j'}|$ for all $j'\in [p]$ (if more than one $j$ satisfies this, choose any of them).  
Since $a_k\ge a_{k-1}\ge a_{k-2}\ge p+1$, by the pigeonhole principle, we have $|C_i|\ge 2$ for any $i\in \{k-2, k-1, k\}\setminus \{l\}$.
Thus both $(U_l\cap V_{i_1}) \cup \bigcup_{j\neq l} C_j$ and $(U_l\cap V_{i_2}) \cup \bigcup_{j\neq l} C_j$ contain $K^{(k)}(1,\dots,1,2,2)$ as a subgraph.
Since $G$ is $K^{(k)}(1,\dots,1,2,2)$-free, both copies of $K^{(k)}(1,\dots,1,2,2)$ contain an edge in $H_p$.
The index vectors of these two edges can be written as $\bfv + \bfu_{i_1}$ and $\bfv + \bfu_{i_2}$ for some $\bfv$, contradicting~($\dagger$).
%Furthermore, by ($\dagger$), either $(U_l\cap V_i) \cup \bigcup_{j\neq l} C_j$ or $(U_l\cap V_j) \cup \bigcup_{j\neq l} C_j$ spans a complete $k$-partite $k$-graph with all edges whose index vector not in $L$.
%Therefore, we conclude that $G$ contains a copy of $K^{(k)}(1,\dots,1,2,2)$, a contradiction.

Since $G$ is $K$-free, each copy $K_1$ of $K$ in $H$ must contain some edge $e$ in $H_p$.
%Suppose $H$ contains a copy of $K$ such that some $e\in E(K)$ is in $H_p$.  
Since $K$ is a blow-up of $e$,  ($\ddagger$)  implies that $\bfi_\cP(e')= \bfi_\cP(e)\in L \pmod{p}$ for all $e'\in E(K_1)$.
Thus, all copies of $K$ in $H$ are actually in $H_p$.
Therefore, $H$ has no $K$-factor by Proposition~\ref{pro:Myc}. 
%Thus all the edges of $K$ can be found in $H_p$. This implies that every copy of $K$ in $H$ is either completely in $H_p$ or completely in $G$, which is impossible. If $H$ contains a $K$-factor, then all copies of $K$ are indeed in $H_p$. This contradicts Proposition~\ref{pro:Myc}. 
\end{proof}

When $K=K^{(k)}(a_1,\dots, a_k)$ is type $d\ne 1$, Proposition~\ref{disprove2} leaves out the following unsettled cases for Conjecture~\ref{conj:Mycroft}: $K^{(k)}(2,\dots, 2, 2s, 2t)$ for $t\ge s\ge 1$ (type 0), and $K^{(k)}(a_1,\dots, a_1, a_{k-1}, a_k)$ (type $d\ge 2$). To see how to derive the second case, we note that if $K$ is type $d\ge 2$ and $p$ is the smallest prime factor of $d$, then $a_1\ge 1$ and $a_{k-2}\le p$ force that $a_1= a_{k-2}$.

\section*{Acknowledgment}

The authors thank Dhruv Mubayi for helpful discussions.
The authors also thank an anonymous referee for his/her many valuable comments that improve the presentation of the paper. In particular, the referee suggested to find more counterexamples to Conjecture~\ref{conj:Mycroft} than what we had in an earlier version of the paper. We followed this comment and derived Proposition~\ref{disprove2} eventually.

\bibliographystyle{abbrv}
\bibliography{Sep2016}

\end{document}